\documentclass{amsart}
 
\usepackage{amsmath}
\usepackage{amsthm,amssymb,color,comment}
\usepackage{bbm}
\usepackage{hyperref}
\usepackage[TS1,T1]{fontenc}
\usepackage[latin1]{inputenc}
\usepackage{dsfont}
\usepackage{tikz}
\usepackage{enumitem}
\usepackage[sort]{cite}

\usetikzlibrary[angles,quotes] 

\numberwithin{equation}{section}

\newtheorem{thm}{Theorem}[section]

\newtheorem{lem}[thm]{Lemma}
\newtheorem{cor}[thm]{Corollary}

\newtheorem{Def}[thm]{Definition}

\theoremstyle{definition}
\newtheorem{Ass}[thm]{Assumption}

\DeclareMathOperator{\dist}{dist}

\newcommand{\R}{\mathbb{R}}
\newcommand{\N}{\mathbb{N}}

\newcommand{\diff}{\mathop{}\!\mathrm{d}}

\makeatletter
\newcommand{\doublewidetilde}[1]{{%
  \mathpalette\double@widetilde{#1}%
}}
\newcommand{\double@widetilde}[2]{%
  \sbox\z@{$\m@th#1\widetilde{#2}$}%
  \ht\z@=.9\ht\z@
  \widetilde{\box\z@}%
}
\makeatother
\author{Miroslav Bul\'{\i}\v{c}ek}
\address{{\it Miroslav Bul{\'\i}{\v{c}}ek:} Mathematical Institute, Faculty of Mathematics and Physics, Charles University, Sokolovsk\'{a} 83, 186 75, Prague, Czech Republic}
\email{mbul8060@karlin.mff.cuni.cz}
\thanks{Miroslav Bul{\'\i}{\v{c}}ek was supported by the project No. 20-11027X financed by GA\v{C}R}

\author{Piotr Gwiazda}
\address{{\it Piotr Gwiazda:} Institute of Mathematics of Polish Academy of Sciences, Jana i J\k edrzeja \'Sniadeckich 8, 00-656 Warsaw, Poland}
\email{pgwiazda@mimuw.edu.pl}
\thanks{Piotr Gwiazda was supported by National Science Center, Poland through project no. 2018/31/B/ST1/02289.}

\author{Jakub Skrzeczkowski}
\address{{\it Jakub Skrzeczkowski:} Faculty of Mathematics, Informatics and Mechanics, University of Warsaw, Stefana Banacha 2, 02-097 Warsaw, Poland}
\email{jakub.skrzeczkowski@student.uw.edu.pl}
\thanks{Jakub Skrzeczkowski was supported by National Science Center, Poland through project no. 2017/27/B/ST1/01569.\\
\indent All authors are grateful to Iwona Chlebicka and B{\l}a\.zej Miasojedow for fruitful discussions and helpful suggestions.
}

\begin{document}

\title[Absence of Lavrentiev phenomenon]{On a range of exponents for absence of Lavrentiev phenomenon for double phase functionals}

\begin{abstract}
For a class of functionals having the $(p,q)$-growth, we establish an improved range of exponents $p$, $q$ for which the Lavrentiev phenomenon does not occur. The proof is based on a standard mollification argument and Young convolution inequality. Our contribution is two-fold. First, we observe that it is sufficient to regularise only bounded functions. Second, we exploit the $L^{\infty}$ bound on the function rather than the $L^p$ estimate on the gradient. Our proof does not rely on the properties of minimizers to variational problems but it is rather a consequence of the underlying  Musielak-Orlicz function spaces. Moreover, our method works for unbounded boundary data, the variable exponent functionals and vectorial problems. In addition, the result seems to be optimal for $p\le d$.
\end{abstract}


\maketitle

\setcounter{tocdepth}{1}
\tableofcontents

\section{Introduction}
\noindent We consider a class of functionals with the so-called $(p,q)$-growth. The prominent example we have in mind is
\begin{equation}\label{eq:functionalG}
\mathcal{G}(u,\Omega) := \int_{\Omega} |\nabla u(x)|^p \diff x + \int_{\Omega} a(x) \, |\nabla u(x)|^q \diff x.
\end{equation}
Here, $\Omega \subset \R^d$ is a bounded Lipschitz domain, $u: \Omega \to \R$ is an argument of the functional $\mathcal{G}$, $a: \Omega \to [0,\infty)$ is a given nonnegative function and $1 \leq p < q <\infty$ are given numbers. Functional $\mathcal{G}$ is the interesting toy model for studying  minimisation of functionals with the so-called non-standard growth. Indeed, depending on whether $a = 0$ or $a > 0$, $\mathcal{G}$ exhibits either the $p$- or the $q$-growth.

A well-known feature of functional $\mathcal{G}$ is the so-called Lavrentiev phenomenon. For instance, there exists a~function $a \in C^{\alpha}(\overline{\Omega})$ with $\alpha\in (0,1)$, exponents $p$, $q$ fulfilling  $p<d<d+{\alpha}<q$ and boundary data $u_0 \in W^{1,q}(\Omega)$ such that 
\begin{equation}\label{eq:lavr_neg}
\inf_{u \in u_0 + W_0^{1,p}(\Omega)} \mathcal{G}(u,\Omega) < \inf_{u \in u_0 + W_0^{1,q}(\Omega)} \mathcal{G}(u,\Omega).
\end{equation}
On the other hand, it is known that if $q \leq p + \alpha \,\frac{p}{d}$, the Lavrentiev phenomenon does not occur for the toy model \eqref{eq:functionalG}, see~\cite{MR2076158}. Under the additional assumption $u_0 \in L^{\infty}(\Omega)$, the range of exponents has been improved to $q \leq p + \alpha$ \cite[Proposition 3.6, Remark 5]{MR3360738}. The latter work heavily depends on the properties of minimizers and the $L^{\infty}$ bound for the minimizer of the functional \eqref{eq:functionalG} form a nontrivial part of the result in \cite{MR3360738}.

In this paper we prove that neither the assumption $u_0 \in L^{\infty}(\Omega)$ nor any further bound on minimizer is irrelevant for the absence of Lavrentiev phenomenon. More precisely, we prove that one does not observe Lavrentiev phenomenon if
\begin{equation}\label{eq:regime_exponents}
q \leq p + \alpha \, \max\left(1, \frac{p}{d} \right)
\end{equation}
and boundary data $u_0 \in W^{1,q}(\Omega)$. In this case, we have
\begin{equation}\label{eq:lavr_pos}
    \inf_{u \in u_0 + W_0^{1,p}(\Omega)} \mathcal{G}(u, \Omega) = \inf_{u \in u_0 + W_0^{1,q}(\Omega)} \mathcal{G}(u, \Omega) = \inf_{u \in u_0 + C_c^{\infty}(\Omega)} \mathcal{G}(u, \Omega).
\end{equation}
This significantly improves the available results for the case $p < d$. Moreover, our proof is elementary as it is based on a simple regularisation argument together with Young's convolution inequality. In particular, we do not use estimates on minimizers of functional \eqref{eq:functionalG}. consequently, our method easily extends to the vector-valued maps and cover variable-exponent functionals as well, see Section~\ref{sect:var-exponent-example}.

The question of whether \eqref{eq:lavr_neg} or \eqref{eq:lavr_pos} holds true is related to the density of $C_c^{\infty}(\Omega)$ in the Musielak--Orlicz--Sobolev space $W^{1,\psi}_0(\Omega)$ related to the functional \eqref{eq:functionalG}, see \eqref{eq:musor_norm}--\eqref{eq:norm_mus_sobolev} for definitions. In this context, we prove that the density result hold true for $p, q$ satisfying \eqref{eq:regime_exponents} which is again better then so-far known regime of exponents announced in \cite{MR3847479}. 

Let us discuss the result of the paper within the context of previous works related to this topic. The first studies concerning functionals changing their ellipticity rate at each point have been carried out by Zhikov \cite{MR681795, MR1209262, MR1486765, MR1350506}. In particular, in \cite{MR1350506} he observed that it may happen that \eqref{eq:lavr_pos} does not hold, extending thus similar observations made by Lavrentiev \cite{MR1553097} and Mania \cite{mania1934}. Another related direction of research is the regularity of minimizers.  Although the fundamental results for minimizers were obtained by Marcellini \cite{MR969900, MR1094446, MR1240398, MR1401415} more than 20 years ago, it is in fact still an active topic of research, see for instance \cite{MR3360738, MR3294408, MR3775180, MR2377404, MR2898773, MR2291779, 	MR3348922, MR3944281,	MR4074606, MR4149062, MR4175825}.

Going back to the functional \eqref{eq:functionalG}, the available results for boundary data $u_0 \in W^{1,q}(\Omega)$ provide both positive and negative answers to the question whether \eqref{eq:lavr_pos} holds true. On the one hand, if $q \leq p + \frac{p\,\alpha}{d}$ then \eqref{eq:lavr_pos} is indeed valid \cite{MR3918367}. On the other hand, if $q > p + \alpha \, \mbox{max}\left(1, \frac{p-1}{d-1}\right)$ then counterexample in \cite[Theorem 34]{MR4153906} shows that \eqref{eq:lavr_pos} is violated (see also \cite[Lemma 7]{MR2076158} for a weaker result concerning the case $p < d < d +{\alpha} < q$ obtained with more elementary methods). In this paper we establish \eqref{eq:lavr_pos} for $q \leq p + \alpha \,\max{\left(1, \frac{p}{d}\right)}$ which partially fills the gap between currently known positive and negative results concerning the Lavrentiev phenomenon. Moreover, in view of~\cite[Theorem 34]{MR4153906}, our result is sharp for $p \leq d$.

Next, we wish to address two issues that appeared in previous papers on this topic. First, in \cite[Lemma 4.1]{MR3294408} there is the following claim: for every $\varepsilon > 0$ and ball $B_r(x) \subset \Omega$, there exists $p_{\varepsilon} < q_{\varepsilon}$ satisfying 
\begin{equation}\label{eq:wrong_example_mingione}
\varepsilon \, p_{\varepsilon} > q_{\varepsilon} - p_{\varepsilon} - \alpha_{\varepsilon} \, \frac{p_{\varepsilon}}{d} > 0,
\end{equation}
a coefficient $a_{\varepsilon} \in C^{\alpha}(\overline{\Omega})$ and a boundary data $u_0 \in W^{1,q}(B_r(x))\cap L^{\infty}(B_r(x))$ such that
$$
\inf_{u \in u_0 + W_0^{1,p_{\varepsilon}}(B_r(x))} \mathcal{G}(u, \Omega) < \inf_{u \in u_0 + W_0^{1,p_{\varepsilon}}(B_r(x)) \cap W^{1,q_{\varepsilon}}_{\text{loc}}(B_r(x))} \mathcal{G}(u, \Omega).
$$
Although it is a very nice result, it does not prove that range of exponents $q \leq p + \alpha\,\frac{p}{d}$ is optimal for absence of the Lavrentiev phenomenon and it does not contradict our result about the range stated in \eqref{eq:regime_exponents}. In fact, authors refer to the counterexample from~\cite{MR2076158} constructed for exponents satisfying $p < d < d+{\alpha} < q$ i.e. exponents that do not meet our range because the distance between $p$ and $q$ is greater than $\alpha$. In fact, it is shown that there exists $p_{\varepsilon}$ and $q_{\varepsilon}$ but it follows also from the proof that they are constructed in the following way: for  $\delta > 0$ to be specified later, we define $p_{\varepsilon}:=d-\delta$, $q_{\varepsilon}:= d+{\alpha} + \delta$ and find a proper counterexample constructed in~\cite{MR2076158}. Then, when $p_{\varepsilon} \geq 1$, we have
$$
\varepsilon \,  p_{\varepsilon} \geq \varepsilon, \qquad q_{\varepsilon} - p_{\varepsilon} - \alpha_{\varepsilon} \, \frac{p_{\varepsilon}}{d} = 2\,\delta + \alpha \, \frac{\delta}{d} = \delta \, \left(2 + \frac{\alpha}{d} \right)
$$
so that \eqref{eq:wrong_example_mingione} is satisfied if we let $\delta := \frac{\varepsilon}{2\,\left(2 + \alpha/d\right)}$. Consequently, $p_{\varepsilon} \to d$ as $\varepsilon\to 0$, which is in perfect coincidence with \eqref{eq:regime_exponents}.

Second, we also want compare our result with~\cite{MR3360738}, where authors proved that the Lavrentiev phenomenon is not observed for $q \leq p + {\alpha}$ in the particular cases when minimizers of \eqref{eq:functionalG} are bounded, but this requires an extra assumption on the boundary data, namely that the boundary data $u_0$ is bounded and apply the maximum principle \cite{MR2197286}. In addition, reasoning in \cite{MR3360738} is based on the so-called Morrey type estimate on the gradient of minimizer which is not an obvious result itself. Comparing to our work, we prove that the Lavrentiev phenomenon does not occur independently of the properties of minimizers or boundedness of boundary data. Our methods are elementary and are based on simple estimates on convolutions. We point out that one could naively think that our result is a consequence of \cite{MR3360738} and a simple approximation argument (boundary data $u_0 \in W^{1,q}(\Omega)$ is approximated with a sequence $\{u_{0,n}\}_{n \in \N} \subset W^{1,q}(\Omega) \cap L^{\infty}(\Omega)$) but it is not necessarily true that sequence of minimizers has then a subsequence converning again to a minimizer of the limit problem.

Finally, we want to point out and emphasize the main novelties of the paper. Standard methods \cite{MR3918367,MR3294408} for proving \eqref{eq:lavr_pos} are based on regularization of arbitrary function $u \in W_0^{1,p}(\Omega)$ satisfying $\mathcal{G}(u,\Omega) < \infty$ with a sequence of smooth functions $u^{\varepsilon} = u \ast \eta_{\varepsilon}$ and passing to the limit $\mathcal{G}(u^{\varepsilon}, \Omega) \to \mathcal{G}(u, \Omega)$ as $\varepsilon \to 0$. The latter is not trivial because the integrand in \eqref{eq:functionalG} is $x$-dependent. Therefore, one approximates locally the integrand with function that does not depend on $x$, see Lemma~\ref{lem:existence_of_minimizer}. This approximation requires good estimate on $\left\|\nabla u^{\varepsilon}\right\|_{\infty}$ which results in constraint on exponents $p$ and $q$. The estimate on gradient is obtained by writing $\nabla u^{\varepsilon} = \nabla u \ast \eta_{\varepsilon}$ and using the fact that $\nabla u \in L^p(\Omega)$. Our main contribution is an observation that it is sufficient to approximate only bounded functions $u$ (i.e. $u \in L^{\infty}(\Omega)$). It turns out that for $p < d$, it is more optimal to write $\nabla u^{\varepsilon} = u \ast \nabla\eta_{\varepsilon}$ and exploit the estimate $u \in L^{\infty}(\Omega)$ rather that $\nabla u \in L^p(\Omega)$. We remark that these observations have been already used in our recent paper on parabolic equations \cite{MR4252140} but at that point we did not observe that similar ideas may bring new information to analysis of the Lavrentiev phenomenon.

The structure of the paper is as follows. In Section~\ref{sect:mainresult} we present the main result, Theorem~\ref{thm:main_theorem}. The theorem holds true under rather complicated assumption so in Section~\ref{sect:examples} we discuss two representative examples. In Section~\ref{sect:mus-orlicz} we review the most important properties of the Musielak--Orlicz--Sobolev spaces. We explain here why it is sufficient to approximate only bounded functions, see Lemmas~\ref{lem:density_of_bounded_in_musielak}~and~\ref{lem:density_implies_lavr}. Then, in Section~\ref{sect:special_case} we present the proof of the main result in the particular case of functional $\mathcal{G}$ as in \eqref{eq:functionalG} and $\Omega = B$ (i.e. a unit ball). In this case we may neglect many technical difficulties and clearly present main ideas. Section~\ref{sect:main_result_general} is devoted to the proof of Theorem~\ref{thm:main_theorem} in the general case. Finally, in Section~\ref{sect:var-exp} we briefly discuss how to extend our work to the case of vectorial problems.

\section{Main result}\label{sect:mainresult}

Let us first set notation. We always assume that $\Omega \subset \R^d$ is a bounded Lipschitz domain and $d$ is the dimension of the space. We write $B$ for the unit open ball centered at $0$. For balls with radius $r$ we use $B_r$ and if the center is at some general point $x$, we write $B_r(x)$ so that $B_1(0) = B$ and $B_r(0) = B_r$. Concerning function spaces, we write $C_c^{\infty}(\Omega)$ for the space of smooth compactly supported functions, $W^{1,p}(\Omega)$ and $W^{1,p}_0(\Omega)$ are usual Sobolev spaces, $W^{1,\psi}(\Omega)$ and $W^{1,\psi}_0(\Omega)$ are the Musielak--Orlicz--Sobolev spaces defined in Section~\ref{sect:mus-orlicz} while $C^{\alpha}(\overline{\Omega})$ is the space of H{\"o}lder continuous functions on $\overline{\Omega}$ with exponent $\alpha\in (0,1]$. Finally, $\eta_{\varepsilon}: \R^d \to \R$ is a usual mollification kernel.

We already introduced the key motivation of the paper, i.e., the functional \eqref{eq:functionalG}, but the main result concern more general cases. We focus in the paper on functionals being of the form
\begin{equation}\label{eq:general_functional}
\mathcal{H}(u, \Omega) = \int_{\Omega} \psi(x, |\nabla u(x)|) \diff x,
\end{equation}
where $\psi$ is the so-called $N$-function and it satisfies the following assumptions:
\begin{Ass}\label{ass:Nfunction}
We assume that an $N$-function $\psi: \Omega \times \R^+ \to \R^+$ satisfies:
\begin{enumerate}[label=(A\arabic*)]
    \item \label{ass:zero} (vanishing at 0) $\psi(x,\xi) = 0$ if and only if $\xi = 0$,
    \item \label{ass:Nf_conv} (convexity) for each $x$, the map $\R^+ \ni \xi \mapsto \psi(x,\xi)$ is convex,
    \item \label{ass:lowerbound} (autonomous lower-bound) there is a strictly increasing and continuous function $m_{\psi}: \R^+ \to \R^+$ such that $m_{\psi}(0) = 0$, $m_{\psi}(\xi) \leq \psi(x,\xi)$ and $\frac{m_{\psi}(\xi)}{\xi} \to  \infty$ as $\xi \to \infty$,
    \item \label{ass:Nf_pq} ($p-q$ growth) there exist exponents $1 < p < q < \infty$ and $\xi_0 \geq 1$ and constants $C_1$ and $C_2$ such that
    $$
    C_1\, |\xi|^p  \leq \psi(x,\xi)  \mbox{ for } \xi \geq \xi_0, \qquad  \psi(x,\xi) \leq C_2 \, (1+|\xi|^q) \mbox{ for all } \xi \geq 0,
    $$
    \item \label{ass:Nf_delta2} ($\Delta_2$ condition) there exists a constant $C_4$ such that
    $$
    \psi(x, 2\xi) \leq C_4 \, \psi(x,\xi).
    $$
\end{enumerate}
\end{Ass}
\begin{Ass}\label{ass:continuity_abstract}
Let $\psi$ be an $N$-function satisfying Assumption \ref{ass:Nfunction}. We assume that for all $D>1$, there are constants $M = M(p,q,D)$ and $N = N(p,q,D)$ such that
$$
\psi(z,\xi) \leq M \, 
\psi(y,\xi) + N
$$
for all balls $B_{\gamma}(x)$, all $y, z \in \overline{B_{\gamma}(x)} \cap \overline{\Omega}$, all $\xi \in \left[0, D\gamma^{-\min\left(1,\, \frac{d}{p}\right)}\right]$ and all $\gamma \in \left(0,\frac{1}{2}\right)$.
\end{Ass}
Let us make few comments on Assumptions~\ref{ass:Nfunction} and~\ref{ass:continuity_abstract}. Conditions \ref{ass:zero}--\ref{ass:lowerbound} are standard in the theory of Orlicz spaces while \ref{ass:Nf_pq} reflects growth of the $N$-function being trapped between $p-$ and $q-$growth. Condition \ref{ass:Nf_delta2} ensures good functional analytic properties in $W_0^{1,\psi}(\Omega)$ cf. Lemma \ref{lem:conv_mus_orlicz_spaces_Delta2}. Assumption~\ref{ass:Nfunction} thus reflects the basic functional setting. The real cornerstone of the paper is however Assumption \ref{ass:continuity_abstract}. It is in fact an abstractly formulated condition on continuity of $\psi$. To understand it better, we note  that it is always possible to estimate for all $x\in \Omega$
$$
 \inf_{y \in \overline{B_{\gamma}(x)} \cap \overline{\Omega}} \psi(y,\xi)
 \leq \psi(x,\xi).
$$
Assumption~\ref{ass:continuity_abstract} states that the above  estimate can be inverted (with a suitable constant). As it seems to be hard to verify it directly, we provide two model examples of $N$-functions $\psi$ satisfying this condition in Section~\ref{sect:examples}. Nevertheless, we would like to emphasize that the prototypic functional \eqref{eq:functionalG} satisfying \eqref{eq:regime_exponents} fulfils also Assumption~\ref{ass:Nfunction}.

The main result of this paper reads:
\begin{thm}\label{thm:main_theorem}
Suppose that $p \leq q + \alpha \max\left(1, \frac{p}{d} \right)$. Let $\mathcal{H}$ be a functional defined with \eqref{eq:general_functional} with $\psi$ satisfying Assumptions \ref{ass:Nfunction}, \ref{ass:continuity_abstract}. Then, for all $u_0 \in W^{1,q}(\Omega)$ we have
$$
    \inf_{u \in u_0 + W_0^{1,p}(\Omega)} \mathcal{H}(u,\Omega) = \inf_{u \in u_0 + W_0^{1,q}(\Omega)} \mathcal{H}(u,\Omega) = \inf_{u \in u_0 + C_c^{\infty}(\Omega)} \mathcal{H}(u,\Omega).
$$
Moreover, space $C_c^{\infty}(\Omega)$ is dense in the Musielak--Orlicz--Sobolev space $W^{1,\psi}_0(\Omega)$.
\end{thm}

\section{Examples of \texorpdfstring{$N$}{N}-functions satisfying Assumption \ref{ass:continuity_abstract}}\label{sect:examples}

\subsection{Standard double phase functionals} In this section, we prove that the $N$-function 
$$
\varphi(x,\xi) = |\xi|^p + a(x) \, |\xi|^q. 
$$
satisfies Assumption \ref{ass:continuity_abstract} provided that $a \in C^{\alpha}(\overline{\Omega})$ and $q \leq p +\alpha\,\max\left(1,\, \frac{p}{d} \right)$. The related functional reads
$$
\mathcal{G}(u,\Omega) := \int_{\Omega} |\nabla u(x)|^p + a(x) \, |\nabla u(x)|^q \diff x. 
$$
To show it, we use the following lemma, whose assumptions are evidently satisfied by the example given above. 

\begin{lem}\label{lem:example1}
Suppose that $\psi$ satisfies Assumption~\ref{ass:Nfunction} with exponents $p$ and $q$. Moreover, assume that there is $\alpha \in (0,1]$ and constant $C_3$ such that for all $x_1, x_2 \in \Omega$ and $\xi \geq \xi_0$ we have
\begin{equation}\label{eq:Holder_cont_functional}
    \left|\psi(x_1, \xi) - \psi(x_2,\xi)\right| \leq C_3\, |x_1 - x_2|^{\alpha} \, (1 + |\xi|^q).
\end{equation}
Then, $\psi$ satisfies Assumption~\ref{ass:continuity_abstract} provided that $q \leq p + \alpha \, \max\left(1, \frac{p}{d}\right)$.
\end{lem}


\begin{proof}[Proof of Lemma~\ref{lem:example1}.]
First, we may assume that $\xi > \xi_0$ as for $\xi \in [0,\xi_0]$ we have
\begin{equation}\label{eq:ver_ass_for_xi_small}
\psi(x,\xi) \leq C_2 \, (1+|\xi|^q) \leq C_2(1+|\xi_0|^q) + \psi(y,\xi)
\end{equation}
so the assertion follows with $M=1$ and $N=C_2 \, (1+|\xi_0|^q)$. Hence, we fix $\xi > \xi_0$ and some ball $B_{\gamma}(x)$ such that $B_{\gamma}(x) \cap \overline{\Omega}$ is not empty. Thanks to \eqref{eq:Holder_cont_functional}, we have for all $y, z \in \overline{B_{\gamma}(x)} \cap \overline{\Omega}$:
$$
\psi(z,\xi) \geq \psi(y,\xi) - C_3\,\left(1 + |\xi|^{q}\right) \,  |y - z|^{\alpha}  \geq \psi(y,\xi) - C_3\,\left(1 + |\xi|^{q}\right)  \, \gamma^{\alpha}.
$$
As $\xi \geq \xi_0 \geq 1$ we have in fact
\begin{equation}\label{eq:first_lower_estimate}
\psi(z,\xi)  \geq \psi(y,\xi) - 2\, C_3\,|\xi|^{q} \, \gamma^{\alpha}.
\end{equation}
To bootstrap this estimate, we fix $\delta \in (0,1)$ and write
\begin{equation}\label{eq:step1}
\psi(z,\xi) = \delta \, \psi(z,\xi) + (1-\delta) \, \psi(z,\xi) \geq \delta \, \psi(y,\xi) - \delta \, 2\, C_3 \, |\xi|^{q} \, \gamma^{\alpha} + (1-\delta) \, C_1 |\xi|^p,
\end{equation}
where we used \eqref{eq:first_lower_estimate} to estimate the first term and lower bound $\psi(z,\xi) \geq C_1 \, |\xi|^p$ to estimate the second term. Now, we may write
\begin{equation}\label{eq:approx_crucial_step}
2\, \delta \,C_3\, |\xi|^{q} \, \gamma^{\alpha} = 2\, \delta \, C_3\, |\xi|^{q-p} \,|\xi|^{p} \, \gamma^{\alpha} \leq 2\,\delta \, C_3 \, D^{q-p} \,  \gamma^{\alpha - (q - p)\,\min\left(1,\, \frac{d}{p}\right)} \, |\xi|^p.
\end{equation}
where we used $|\xi| \leq D\,\gamma^{-\min\left(1,\, \frac{d}{p}\right)}$. As $q - p \leq \alpha \, \max\left(1, \frac{p}{d}\right)$, we have
$$
\alpha - (q - p)\,\min\left(1,\, \frac{d}{p}\right) \geq \alpha - \alpha\,  \max\left(1, \frac{p}{d}\right) \, \min\left(1,\, \frac{d}{p}\right) = \alpha - \alpha = 0.
$$
It follows that
$
\gamma^{\alpha - (q - p)\,\min\left(1,\, \frac{d}{p}\right)} \leq 1
$ for $\gamma \in \left(0, \frac{1}{2}\right)$. Hence, coming back to \eqref{eq:step1} we obtain
\begin{multline*}
\psi(z,\xi) \geq \delta \, \psi(y,\xi) - \delta \, 2\, C_3\, D^{q-p} \, |\xi|^p+ (1-\delta) \, C_1 \, |\xi|^p = \\ = \delta \, \psi(y,\xi)  + \left( (1 - \delta)\,C_1 - \delta \,2\, C_3\,D^{q-p} \right)\, |\xi|^p.
\end{multline*}
We choose $\delta = \frac{C_1}{C_1 + 2\,C_3\,D^{q-p}}$ so that $\left((1 - \delta)\,C_2 - \delta \, C_3\,D^{q-p} \right)\, |\xi|^p = 0$. Hence, for all $y, z \in \overline{B_{\gamma}(x)} \cap \overline{\Omega}$
$$
\psi(z,\xi) \geq \delta \, \psi(y,\xi)
$$
so combining with \eqref{eq:ver_ass_for_xi_small}, the proof is concluded with $M = \max\left(1/\delta, 1\right)$ and $N = C_2\left(1+|\xi_0|^q\right)$. 
\end{proof}

\subsection{Variable exponent double phase functionals}\label{sect:var-exponent-example}
In this section we prove that $N$-function
\begin{equation}\label{eq:var_exp_Nfun}
\phi(x,\xi) := |\xi|^{p(x)} +a(x)\,|\xi|^{q(x)}
\end{equation}
satisfies our Assumption \ref{ass:continuity_abstract}. The related functional reads:
\begin{equation}\label{eq:var_exp_fun}
\mathcal{J}(u,\Omega):= \int_{\Omega} \left[|\nabla u|^{p(x)} +a(x)\,|\nabla u|^{q(x)}\right] \diff x.
\end{equation}
\begin{Ass}\label{ass:var_exp}
We assume that:
\begin{enumerate}[label=(B\arabic*)]
\item ($p-q$ growth) there exist $p, q$ with $1 < p \leq q$ such that the functions $p(x), q(x): \Omega \to [1, \infty)$ satisfy $p \leq p(x) \leq q(x) \leq q$,
\item ($\log$-H{\"o}lder continuity) there are constants $C_p, C_q$ such that for all $x, y \in \Omega$ with $|x-y| \leq \min\left(\mbox{diam}\, \Omega, \frac{1}{2} \right)$ we have
$$
|p(x) - p(y)| \leq -\frac{C_p}{\log |x-y|}, \qquad \qquad |q(x) - q(y)| \leq -\frac{C_q}{\log |x-y|}.
$$
\item ($\alpha$-H{\"o}lder continuity) $a \in C^{\alpha}(\overline{\Omega})$ with constant $|a|_{\alpha}$.
\end{enumerate}
\end{Ass}

\begin{lem}\label{lem:infimum_lem_var_exp}
Under Assumption \ref{ass:var_exp}, $N$-function $\phi$ defined with \eqref{eq:var_exp_Nfun} satisfies Assumption \ref{ass:continuity_abstract} for $q$ and $p$ such that $q \leq p + \alpha \, \max\left(1, \frac{p}{d}\right)$.
\end{lem}
\begin{proof}
As in the proof of Lemma \ref{lem:example1}, we only need to consider $\xi \geq 1$. Let us estimate $\frac{\phi(x,\xi)}{\phi(y,\xi)}$ for $x, y \in \Omega$ such that $|x-y| \leq \min\left(\mbox{diam}\, \Omega, \frac{1}{2} \right)$. We have
\begin{equation}\label{eq:quotient_estimate}
\begin{split}
\frac{\phi(x,\xi)}{\phi(y,\xi)} &= \frac{|\xi|^{p(x)} +a(x)\,|\xi|^{q(x)}}{|\xi|^{p(y)} +a(y)\,|\xi|^{q(y)}} = \frac{|\xi|^{q(x)}}{|\xi|^{q(y)}} \frac{|\xi|^{p(x)-q(x)} +a(x)}{|\xi|^{p(y)-q(y)} +a(y)} \leq \\
&\leq |\xi|^{q(x)-q(y)}\, \left[\frac{|\xi|^{p(x)-q(x)}}{|\xi|^{p(y)-q(y)}} + \frac{a(x) - a(y)}{|\xi|^{p(y)-q(y)}} + 1 \right] \\
&\leq |\xi|^{q(x)-q(y)}\, \left[{|\xi|^{p(x)-p(y)}}\,{|\xi|^{q(y)-q(x)}} + \frac{a(x) - a(y)}{|\xi|^{p(y)-q(y)}} + 1 \right]\\
&\leq |\xi|^{-\frac{C_q}{\log|x-y|}}\, \left[{|\xi|^{-\frac{C_p}{\log|x-y|}}}\,{|\xi|^{-\frac{C_q}{\log|x-y|}}} + {|a|_{\alpha}\,|x-y|^{\alpha}} \, {|\xi|^{q(y)-p(y)}} + 1 \right] \\
&\leq |\xi|^{-\frac{C_q}{\log|x-y|}}\, \left[{|\xi|^{-\frac{C_p}{\log|x-y|}}}\,{|\xi|^{-\frac{C_q}{\log|x-y|}}} + {|a|_{\alpha}\,|x-y|^{\alpha}} \, {|\xi|^{\alpha\,\max\left(1,\, \frac{p}{d} \right) }} + 1 \right]
\end{split}
\end{equation}
Now, let $D> 1$ and $\gamma \in \left(0,\frac{1}{2}\right)$. Suppose that $q \leq p + \alpha\, \max\left(1,\, \frac{p}{d}\right)$, $|x-y| \leq \gamma$ and $\xi \in \left[1, D\gamma^{-\min\left(1,\, \frac{d}{p}\right)}\right]$. Then, 
$$
|\xi|^{-\frac{C_q}{\log|x-y|}} \leq \left(D\gamma^{-\min\left(1,\, \frac{d}{p}\right)} \right)^{-\frac{C_q}{\log \gamma}} = D^{-\frac{C_q}{\log \gamma}} \, \gamma^{\frac{C_q}{\log \gamma}\,\min\left(1,\, \frac{d}{p}\right)} = D^{-\frac{C_q}{\log \gamma}} \, e^{C_q\, \min\left(1,\, \frac{d}{p}\right)}.
$$
It follows that there is a numerical constant $E$ such that $D^{-\frac{C_q}{\log \gamma}} \, e^{C_q}, D^{-\frac{C_p}{\log \gamma}} \, e^{C_p} \leq E$ for all $\gamma \in \left(0,\frac{1}{2}\right)$. Using \eqref{eq:quotient_estimate} we obtain
$$
\frac{\phi(x,\xi)}{\phi(y,\xi)} \leq E\,\left(E^2 + |a|_{\alpha}\,\gamma^{\alpha} \, \left(D\gamma^{-\min\left(1,\, \frac{d}{p}\right)} \right)^{\alpha \, \max\left(1,\, \frac{p}{d} \right)} + 1 \right) = E\,\left(E^2 + D^{\alpha \, \max\left(1,\, \frac{p}{d} \right)}\,|a|_{\alpha} + 1 \right) =:M,
$$
where we used $\max\left(1,\, \frac{p}{d} \right) \, \min\left(1,\, \frac{d}{p}\right)= 1$ in the last line. We deduce
$$
\phi(x,\xi) \leq M \, 
\phi(y,\xi).
$$
\end{proof}

\section{Musielak--Orlicz--Sobolev spaces}\label{sect:mus-orlicz}

Our results are based on smooth approximation in the Musielak--Orlicz spaces, so we first recall their definitions and basic properties. For more details, we refer to monographs \cite{MR3931352, chlebicka2019book}. We consider an $N$-function $\psi:\Omega \times \R^+ \to \R$ satisfying \ref{ass:zero}--\ref{ass:Nf_delta2} in Assumption~\ref{ass:Nfunction}. We define the related Luxembourg norm with
\begin{equation}\label{eq:musor_norm}
\| f \|_{\psi} = \inf \left\{\lambda>0: \int_{\Omega} \psi\left(x, \frac{\left|f(x)\right|}{\lambda}\right) \diff x \leq 1 \right\}.
\end{equation}
Finally, the Musielak--Orlicz--Sobolev spaces are defined as
\begin{equation}\label{eq:orsob_spaces}
W^{1,\psi}(\Omega) = \{w \in W^{1,1}(\Omega): \|\nabla w \|_{\psi} < \infty \}, \qquad W^{1,\psi}_0(\Omega) = W^{1,1}_0(\Omega) \cap W^{1,\psi}(\Omega),
\end{equation}
the latter one corresponds to the space of functions vanishing at the boundary. These are normed spaces with norm
\begin{equation}\label{eq:norm_mus_sobolev}
\|u\|_{1,\psi} = \|u\|_{1} + \|\nabla u\|_{\psi}.
\end{equation}
One can think of $W^{1,\psi}(\Omega)$ as the space of functions having gradient integrable with $p$ or $q$ power depending on whether $a = 0$ or not. 

We summarize some properties of the Musielak--Orlicz spaces in the following lemma. They are mainly consequences of~\ref{ass:Nf_delta2} in Assumption~\ref{ass:Nfunction}. The proof can be found in many texts on Orlicz spaces \cite{MR3931352, chlebicka2019book}, yet for the sake of completeness we present the proof in Appendix~\ref{app:proof_of_orlicz_lemma}.

\begin{lem}\label{lem:conv_mus_orlicz_spaces_Delta2}
Let $\psi$ satisfy Assumption \ref{ass:Nfunction}. Then,
\begin{enumerate}[label=(C\arabic*)]
    \item\label{item:mus-or:1} $\|f\|_{\psi} < \infty \iff \int_{\Omega} \psi(x,c|f(x)|) \diff x < \infty$ for some $c>0$ $\iff$ $\int_{\Omega} \psi(x,c|f(x)|) \diff x < \infty$ for all $c >0$,
    \item\label{item:mus-or:2} $\|f_n - f\|_{\psi} \to 0 \iff$ for some $c>0$  $\int_{\Omega} \psi(x,c\,|f_n(x) - f(x)|) \diff x \to 0$ $\iff$ for all $c >0$ $\int_{\Omega} \psi(x,c|f_n(x) - f(x)|) \diff x \to 0$,
    \item\label{item:mus-or:3} if $\|f\|_{\psi} < \infty$ and any of the conditions in (B) is satisfied, we have $\int_{\Omega} \psi(x,|f_n(x)|) \diff x \to \int_{\Omega} \psi(x,|f(x)|) \diff x$,
    \item\label{item:mus-or:4} if $f_n \to f$ a.e. on $\Omega$, $\|f\|_{\psi} < \infty$ and the sequence $\{\psi(x,|f_n(x)|\}_{n \in \N}$ is uniformly integrable then $\|f_n - f\|_{\psi} \to 0$,
    \item\label{item:mus-or:5} if $\|f\|_{\psi} < \infty$ then $f \in L^1(\Omega)$.
\end{enumerate}
\end{lem}

Next two lemmas show that to prove the absence of the Lavrentiev phenomenon, it is sufficient to demonstrate that every $u\in W^{1,\psi}_0(\Omega) \cap L^{\infty}(\Omega)$ can be approximated in the topology of $W^{1,\psi}$ by smooth function from  $C_c^{\infty}(\Omega)$.

First lemma shows, that it is enough to consider only bounded functions. Notice that we do not impose any specific assumption on the $N$-function $\psi$ here.
\begin{lem}\label{lem:density_of_bounded_in_musielak}
Space $W_0^{1,\psi}(\Omega) \cap L^{\infty}(\Omega)$ is dense in $W_0^{1,\psi}(\Omega)$.
\end{lem}
\begin{proof}
Let $u \in W_0^{1,\psi}(\Omega)$. Consider truncation of $u$ defined as
\begin{equation}\label{eq:trun_def}
T_k(u) = \begin{cases} 
u  & \mbox{ if } |u| \leq k, \\
k\, \frac{u}{|u|} & \mbox{ if } |u| > k.
\end{cases}
\end{equation}
Clearly, $T_k(u) \in L^{\infty}(\Omega)$. Moreover, chain rule for Sobolev maps implies that $\nabla T_k(u) = \nabla u \, \mathds{1}_{|u|\leq k}$ so that $\nabla T_k(u) \to \nabla u$ a.e. as $k \to \infty$. As $\psi(x,\xi) = 0$ if and only if $x = 0$, we have
$$
0 \leq \psi(x, \left|\nabla T_k(u)\right|) = \psi(x, \left|\nabla u\right|) \, \mathds{1}_{|u| \leq k} \leq \psi(x,\left|\nabla u\right|) 
$$
so that the sequence $\left\{\psi(x, \left|\nabla T_k(u))\right| \right\}_{k \in \N}$ is uniformly integrable. Application of~\ref{item:mus-or:4} from Lemma~\ref{lem:conv_mus_orlicz_spaces_Delta2}  concludes the proof.
\end{proof}

\begin{lem}\label{lem:density_implies_lavr}
Suppose that $\psi$ satisfies \ref{ass:zero}--\ref{ass:Nf_delta2} in Assumption \ref{ass:Nfunction} and that for every $u\in W_0^{1,\psi}(\Omega) \cap L^{\infty}(\Omega)$ there exists a sequence $\{u^n\}_{n=1}^{\infty} \subset C_c^{\infty}(\Omega)$ such that $\|u^n-u\|_{1,\psi} \to 0$ as $n\to \infty$. Then, the space $C_c^{\infty}(\Omega)$ is dense in $W_0^{1,\psi}(\Omega)$ and the Lavrentiev phenomenon does not occur, i.e., for all $u_0 \in W^{1,q}(\Omega)$ we have
$$
\inf_{u\in u_0 + W^{1,p}_0(\Omega)} \mathcal{H}(u, \Omega) =
\inf_{u\in u_0 + W^{1,q}_0(\Omega)} \mathcal{H}(u, \Omega) = \inf_{u\in u_0 + C_c^{\infty}(\Omega)} \mathcal{H}(u, \Omega).
$$
\end{lem}
\begin{proof}
Thanks to Lemma~\ref{lem:density_of_bounded_in_musielak}, $C_c^{\infty}(\Omega)$ is dense in $W_0^{1,\psi}(\Omega)$. Let $u^* \in W^{1,p}(\Omega)$ be the minimizer of $\mathcal{H}$ i.e.
$$
\inf_{u\in u_0 + W^{1,p}(\Omega)} \mathcal{H}(u,\Omega) = \mathcal{H}(u^*,\Omega).
$$
The minimizer exists by a usual application of direct method in calculus of variations, cf. \cite[Theorem 2.7]{MR3821514}. Note that we always have
$$
\mathcal{H}(u^*,\Omega) = \inf_{u\in u_0 + W^{1,p}(\Omega)} \mathcal{H}(u,\Omega) \leq
\inf_{u\in u_0 + W^{1,q}(\Omega)} \mathcal{H}(u,\Omega) \leq \inf_{u\in u_0 + C_c^{\infty}(\Omega)} \mathcal{H}(u,\Omega)
$$
because $p < q$. To prove the reversed inequality, we write $u^* = u_0 + \overline{u}$ where $u_0 \in W^{1,q}(\Omega)$ and $\overline{u} \in W_0^{1,p}$. Note that $u_0 \in W^{1,\psi}(\Omega)$ (because $W^{1,q}(\Omega) \subset W^{1,\psi}(\Omega)$) and $u^* \in W^{1,\psi}(\Omega)$ (because $\mathcal{H}(u^*,\Omega) < \infty$ cf. Lemma \ref{lem:conv_mus_orlicz_spaces_Delta2} \ref{item:mus-or:1}). It follows that $\overline{u} = u^* - u_0 \in W^{1,\psi}_0(\Omega)$. Now, consider the sequence $\{u_n\}_{n\in \N} \subset C_c^{\infty}(\Omega)$ such that $u_n \to \overline{u}$ in $W^{1,\psi}(\Omega)$ which exists due to the assumptions. It follows that $u_n + u_0 \to \overline{u} + u_0 = u^*$ in $W^{1,\psi}(\Omega)$. In particular, $\mathcal{H}(u_0 + u_n,\Omega) \to \mathcal{H}(u^*,\Omega)$ cf. Lemma~\ref{lem:conv_mus_orlicz_spaces_Delta2} \ref{item:mus-or:3}. Note that $u_0 + u_n \in u_0 + C_c^{\infty}(\Omega)$. It follows that 
$$
\inf_{u\in u_0 + C_c^{\infty}(\Omega)} \mathcal{H}(u,\Omega) \leq \mathcal{H}(u_0 + u_n,\Omega) \to \mathcal{H}(u^*,\Omega) \qquad \mbox{ as } n \to \infty.
$$
\end{proof}


\section{Proof of Theorem \ref{thm:main_theorem} in the special case}\label{sect:special_case}
\noindent In this section we prove Theorem~\ref{thm:main_theorem} in the case when $\Omega = B$ (unit ball centered at $0$) and the $N$-function is defined via the formula
\begin{equation}\label{eq:Nfunction_special_case}
\varphi(x,\xi) = |\xi|^p + a(x) \, |\xi|^q. 
\end{equation}
The corresponding functional then takes the following form
$$
\mathcal{G}(u,B) := \int_B |\nabla u(x)|^p + a(x) \, |\nabla u(x)|^q \diff x. 
$$
Note that, if $a \in C^{\alpha}(\overline{B})$ and $q\leq p + \alpha\, \max\left(1, \,\frac{p}{d}\right)$, it follows from Lemma~\ref{lem:example1} that $\varphi$ satisfies Assumption~\ref{ass:continuity_abstract}.

The main purpose of this section is that we avoid all technical difficulties and focus only on the main parts of the proof. More precisely, we do not need to take care of difficulties coming from:
\begin{itemize}
    \item geometric properties of general Lipschitz domain $\Omega$,
    \item situation when for general $N$-function $\psi$ there is no local minimizer of the map $x \mapsto \psi(x,\xi)$ valid for all values of $\xi$.
\end{itemize}

We start with introducing mollification that will be used to define the approximation.

\begin{Def}[Mollification with squeezing]\label{def:mol_in_sp}
For $\varepsilon \in (0, 1/4)$ we set $\eta_{\varepsilon}(x) = \frac{1}{\varepsilon^d} \eta\left(\frac{x}{\varepsilon}\right)$ where $\eta$ is a usual mollification kernel. Then, for arbitrary $u: \R^d \to \R$, we define $u^{\varepsilon}: \R^d \to \R$ as
$$
u^{\varepsilon}(x) = \int_{\R^d} \eta_{\varepsilon}(y) \, u\left(\frac{x}{1-2 \varepsilon} - y \right) \diff y.
$$
\end{Def}

\begin{lem}\label{lem:geometry_special_case}
Let  $u \in W^{1,1}_0(B)$ and be extended by zero onto $\R^d$. Then, $u^{\varepsilon} \in C_c^{\infty}(B)$. Moreover, $\frac{x}{1-2 \varepsilon} - y  \in B_{5\varepsilon}(x)$ for all $y$ such that $|y|\leq \varepsilon$.
\end{lem}
\begin{proof}
Smoothness follows from standard properties of convolutions cf. \cite[Appendix C.4]{evans1998partial}. To see the compact support, let $|x| \geq 1 - \varepsilon$ and $|y| \leq \varepsilon$. Then,
$$
\left| \frac{x}{1- 2\,\varepsilon} - y \right| \geq \frac{1-\varepsilon}{1 - 2\,\varepsilon} - \varepsilon = \frac{1-\varepsilon}{1 - 2\,\varepsilon} - \frac{\varepsilon-2\varepsilon^2}{1 - 2\,\varepsilon} = 
\frac{1 - 2\,\varepsilon + 2\,\varepsilon^2}{1 - 2\,\varepsilon} = 1 + \frac{2\,\varepsilon^2}{1 - 2\,\varepsilon} > 1
$$
so that $u\left(\frac{x}{1-2 \varepsilon} - y \right) = 0$. It follows that $u^{\varepsilon}$ is supported in $B_{1-\varepsilon}$. To see the second property, we estimate
$$
\left|x - \frac{x}{1-2 \varepsilon} + y \right| \leq |x| \frac{2\varepsilon}{1-2\varepsilon} + |y| \leq 4 \varepsilon + \varepsilon = 5\varepsilon,
$$
where we used $\frac{1}{1 - 2\,\varepsilon} \leq 2$, i.e. $\varepsilon \leq \frac{1}{4}$.
\end{proof}

Before formulating the main theorem of this section, we  state and prove two results: a technical lemma concerning approximating sequence and a simple observation concerning $N$-function $\varphi$. 

\begin{lem}\label{lem:technical_convergence_of_mollification}
Let $u \in W^{1,1}_0(B)$ be such that $\mathcal{G}(u)<\infty$ and consider its extension to $\R^d$. Then,
\begin{enumerate}[label=(D\arabic*)]
    \item \label{item:conv_special_case_1} $\varphi\left(\frac{x}{1-2\,\varepsilon},  \left|\nabla  u\right| \left(\frac{x}{1-2\,\varepsilon}\right) \right) \to \varphi(x,\left|\nabla u(x)\right|)$ in $L^1(\R^d)$,
    \item \label{item:conv_special_case_2} $ \int_{\R^d}  \varphi\left(\frac{x}{1-2\,\varepsilon}-y,  \left| \nabla  u\right| \left(\frac{x}{1-2\,\varepsilon}-y\right)\right) \eta_{\varepsilon}(y) \diff y \to \varphi\left(x, \left| \nabla u\right|(x)  \right)$ in $L^1(\R^d)$.
\end{enumerate}
\end{lem}
\begin{proof}
To see~\ref{item:conv_special_case_1}, we note that the convergence holds in the pointwise sense. Moreover, the considered sequence is supported only for $x \in B_{1-2\varepsilon}$. Therefore, to establish convergence in $L^1(\R^d)$, it is sufficient to prove equiintegrability of the sequence $\left\{\varphi\left(\frac{x}{1-2\,\varepsilon},  \left|\nabla  u  \right| \left(\frac{x}{1-2\,\varepsilon}\right) \right)\right\}_{\varepsilon}$ and apply the Vitali convergence theorem. To this end, we need to prove 
$$
\forall_{\eta>0}\, \exists_{\delta>0} \, \forall_{A\subset B, |A|\leq \delta} \qquad \, \int_{A} \varphi\left(\frac{x}{1-2\,\varepsilon}, \left|\nabla u\right| \left(\frac{x}{1-2\,\varepsilon}\right) \right) \diff x \leq \eta.
$$
We fix $\eta$ and arbitrary $A \subset B$. Using change of variables we have 
$$
\int_{A} \varphi\left(\frac{x}{1-2\,\varepsilon},  \left|\nabla u\right| \left(\frac{x}{1-2\,\varepsilon}\right) \right) \diff x = (1-2\,\varepsilon)^d \int_{A/(1-2\,\varepsilon)} \varphi(x, \left|\nabla u\right|(x) )\diff x \leq \int_{2A} \varphi(x,\left|\nabla u\right|(x)) \diff x, 
$$
where for $c \in \R^+$, $cA$ denotes a usual scaled set. By assumptions we have $\mathcal{G}(u) < \infty$, so that if we set
$$
\omega(\tau) := \sup_{C \subset \R^d: |C|\leq \tau} \int_C \varphi(x,\left|\nabla u\right|(x)) \diff x,
$$
then $\omega(\tau)$ is a non-decreasing function, continuous at 0. Therefore, we may find $\tau$ such that $\omega(\tau) \leq 2^{-q} \, \eta$. Then, we choose $\delta = 2^{-d} \, \tau$ to conclude the proof of~\ref{item:conv_special_case_1}. Finally, the convergence result~\ref{item:conv_special_case_2} follows from Young's convolutional inequality and~\ref{item:conv_special_case_1}.
\end{proof}

\begin{lem}\label{lem:existence_of_minimizer}
Let $\varphi$ be given by \eqref{eq:Nfunction_special_case}. Then for all balls $B_{\gamma}(x)$ such that $\overline{B_{\gamma}(x)} \cap \overline{B}$ is nonempty, there exists $x^* \in \overline{B_{\gamma}(x)} \cap \overline{B}$ such that for all $\xi$
$$
\inf_{y \in \overline{B_{\gamma}(x)} \cap \overline{B}} \varphi(y,\xi) = \varphi(x^*,\xi).
$$
\end{lem}
\begin{proof}
Using continuity of $a$ and compactness of $\overline{B_{\gamma}(x)} \cap \overline{B}$ we have
$$
\inf_{y \in \overline{B_{\gamma}(x)} \cap \overline{B}} \varphi(y,\xi) = \inf_{y \in \overline{B_{\gamma}(x)} \cap \overline{B}} \left[ |\xi|^p + a(y) \, |\xi|^q \right] = |\xi|^p + |\xi|^q \inf_{y \in \overline{B_{\gamma}(x)} \cap \overline{B}} a(y)
$$
and we choose ${y}^*$ such that $\inf_{y \in \overline{B_{\gamma}(x)} \cap \overline{B}} a(y) = a({y}^*)$.
\end{proof}

\begin{thm}[Theorem \ref{thm:main_theorem} in the special case]\label{res:approx_theorem}
Let $u \in W^{1,\varphi}_0(B) \cap L^{\infty}(B)$ with $a \in C^{\alpha}(\overline{B})$. Suppose that 
$$
q \leq p + {\alpha} \max\left( 1, \frac{p}{d}\right).
$$
Consider sequence $u^{\varepsilon}$ as in Definition \ref{def:mol_in_sp} with 
$
\varepsilon \in \left(0, \frac{1}{4} \right).
$
Then,
\begin{enumerate}[label=(E\arabic*)]
\item\label{item_mainthm1_item1} $u^{\varepsilon} \in C_c^{\infty}(B)$,
\item\label{item_mainthm1_item2} $\mathcal{G}\left(u^{\varepsilon},B\right) \to \mathcal{G}(u,B)$ as $\varepsilon \to 0$,
\item\label{item_mainthm1_item3} $u^{\varepsilon} \to u$ in $W^{1,\varphi}(B)$ as $\varepsilon \to 0$,
\item\label{item_mainthm1_item4} $C_c^{\infty}(B)$ is dense in $W^{1,\varphi}_0(B)$ and Lavrentiev phenomenon does not occur, i.e. for all boundary data $u_0 \in W^{1,q}(B)$
\begin{equation*}
    \inf_{u \in u_0 + W_0^{1,p}(B)} \mathcal{G}(u,B) = \inf_{u \in u_0 + W_0^{1,q}(B)} \mathcal{G}(u,B) = \inf_{u \in u_0 + C_c^{\infty}(B)} \mathcal{G}(u,B).
\end{equation*}
\end{enumerate}
\end{thm}

\begin{proof}
The first property follows from construction. To prove convergence, we note that 
$$
\mathcal{G}\left(u^{\varepsilon},B\right) =
\int_B \varphi(x, \left|\nabla u^{\varepsilon}\right|(x)) \diff x.
$$
We would like to take mollification out of the function $\varphi$ using its convexity and Jensen's inequality. However, this is not possible as function $\varphi$ depends also on $x$ explicitly. To overcome this problem, we apply Assumption~\ref{ass:continuity_abstract}, which allows us to approximate the function $\varphi(x,\xi)$ locally by a function depending only on $\xi$. Notice that $\psi$ saisfies Assumption~\ref{ass:continuity_abstract} thanks to Lemma~\ref{lem:example1} and the structural assumption~\eqref{eq:Nfunction_special_case}. \\

\noindent \underline{\textbf{Case 1: $p \leq d$.}} In this case we have $q \leq p + \alpha$. Using Young's convolution inequality we obtain:

\begin{equation}\label{eq:estimate_infinity}
\left\| \nabla u^{\varepsilon}  \right\|_{\infty} \leq \left\| u \right\|_{\infty} \, \left\| \nabla \eta_{\varepsilon} \right\|_{1} \leq D\, (5\varepsilon)^{-1},
\end{equation}
where we choose $D:= 5\,\left\| u \right\|_{\infty} \,\left\| \nabla \eta \right\|_{1}$. Let $x \in B$. Applying Assumption \ref{ass:continuity_abstract} with $\gamma = 5\,\varepsilon$ and Lemma \ref{lem:existence_of_minimizer} we obtain $x^* \in \overline{B_{5\varepsilon}(x)} \cap \overline{B}$ and constants $M$, $N$ such that
\begin{equation}\label{eq:estimate_for_infimum1}
\varphi(x, \left| \nabla u^{\varepsilon} \right|(x) ) \leq 
M\, \varphi(x^*, \left| \nabla u^{\varepsilon}\right| (x)  ) + N.
\end{equation}
Note that
\begin{align*}
\varphi(x^*, \left|\nabla u^{\varepsilon}(x)\right| ) &=  \varphi\left( x^*, \frac{1}{1-2\,\varepsilon} \left|\int_{B_{\varepsilon}} \nabla  u \left(\frac{x}{1-2\,\varepsilon}-y\right)  \, \eta_{\varepsilon}(y) \diff y \right|\right) \leq 
\\
&\leq \left(\frac{1}{1-2\,\varepsilon}\right)^q \, \varphi\left( x^*, \int_{B_{\varepsilon}} \left|\nabla  u\right| \left(\frac{x}{1-2\,\varepsilon}-y\right)   \, \eta_{\varepsilon}(y) \diff y\right) \\
&\leq 2^q \, \varphi\left( x^*,  \int_{B_{\varepsilon}} \left| \nabla  u\right| \left(\frac{x}{1-2\,\varepsilon}-y\right)  \, \eta_{\varepsilon}(y) \diff y\right),
\end{align*}
where we used that $\varphi$ is of the form \eqref{eq:Nfunction_special_case}. Then, Jensen's inequality implies
\begin{equation*}
\varphi\left( x^*,  \int_{B_{\varepsilon}} \left|\nabla  u\right| \left(\frac{x}{1-2\,\varepsilon}-y\right)  \, \eta_{\varepsilon}(y) \diff y\right)
\leq
\int_{B_{\varepsilon}}   \varphi\left(x^*,  \left|\nabla u\right| \left(\frac{x}{1-2\,\varepsilon}-y\right)\right) \eta_{\varepsilon}(y) \diff y.
\end{equation*}
If $\frac{x}{1-2\,\varepsilon}-y$ does not belong to $\overline{B}$ then $\varphi\left(x^*,  \left|\nabla  u\right| \left(\frac{x}{1-2\,\varepsilon}-y\right)\right)  = 0$. Otherwise, Lemma \ref{lem:geometry_special_case} implies $\frac{x}{1-2\,\varepsilon}-y \in \overline{B} \cap \overline{B_{5\,\varepsilon}(x)}$ so that
$$
\varphi\left(x^*,  \left|\nabla  u\right| \left(\frac{x}{1-2\,\varepsilon}-y\right)\right)  \leq
\varphi\left(\frac{x}{1-2\,\varepsilon}-y, \left|\nabla  u\right| \left(\frac{x}{1-2\,\varepsilon}-y\right)\right)
$$
due to the minimality of $x^*$ and nonnegativity of $a$. We conclude
\begin{equation}\label{eq:crucial_estimate_two_UNI_INT}
\varphi(x, \left|\nabla u^{\varepsilon}\right|(x) ) \leq 
2^q \, M\, \int_{B_{\varepsilon}}   \varphi\left(\frac{x}{1-2\,\varepsilon}-y,  \left|\nabla u\right| \left(\frac{x}{1-2\,\varepsilon}-y\right)\right) \eta_{\varepsilon}(y) \diff y + N.
\end{equation}
Now, we observe that $\varphi(x, \left|\nabla u^{\varepsilon}\right|(x) )$ converges a.e. to $\varphi(x, \left|\nabla u\right|(x))$. Moreover, the (RHS) of \eqref{eq:crucial_estimate_two_UNI_INT} is convergent in $L^1(B)$ cf. Lemma \ref{lem:technical_convergence_of_mollification} \ref{item:conv_special_case_2} so that $\left\{\varphi(x, \left|\nabla u^{\varepsilon}\right|(x) ) \right\}_{\varepsilon}$ is uniformly integrable in $L^1(B)$.  Therefore, Vitali convergence theorem implies
$$
\varphi(x, \left|\nabla u^{\varepsilon}\right|(x) ) \to \varphi(x, \left|\nabla u\right|(x)) \qquad \mbox{ in } L^1(B)  \mbox{ as } \varepsilon \to 0.
$$
Thanks to triangle inequality we obtain \ref{item_mainthm1_item2}. To see \ref{item_mainthm1_item3}, we note a simple estimate $|a+b|^q \leq 2^{q-1} \left(|a|^q +|b|^q\right)$ so that
$$
\varphi\left(x,\left|\nabla u(x) - \nabla u^{\varepsilon}(x)\right|\right) \leq 2^{q-1} \varphi\left(x,\left|\nabla u\right|(x) \right) + 2^{q-1}\varphi\left(x, \left|\nabla u^{\varepsilon}\right|(x)\right).
$$
It follows that the sequence $\left\{\varphi\left(x,\left|\nabla u(x) - \nabla u^{\varepsilon}(x)\right|\right)\right\}_{\varepsilon}$ is again uniformly integrable and Vitali convergence theorem yields
$$
\varphi\left(x,\left|\nabla u(x) - \nabla u^{\varepsilon}(x)\right|\right) \to 0 \qquad \mbox{ in } L^1(B) \mbox{ as } \varepsilon \to 0,
$$
concluding the proof of \ref{item_mainthm1_item3}. This shows that any bounded function in $W_0^{1,\varphi}(B)$ can be approximated with smooth compactly supported functions so that \ref{item_mainthm1_item4} follows from Lemma \ref{lem:density_implies_lavr}. \\

\noindent \underline{\textbf{Case 2: $p > d$.}} In this case we have $q \leq p + \alpha\,\frac{p}{d}$. Note that
$$
\nabla u^{\varepsilon}(x) = \frac{1}{1-2\,\varepsilon} \int_{B_{\varepsilon}} \nabla  u \left(\frac{x}{1-2\,\varepsilon}-y\right)  \, \eta_{\varepsilon}(y).
$$
Therefore, instead of \eqref{eq:estimate_infinity}, we can compute
\begin{equation}\label{eq:estimate_infinity2}
\left\| \nabla u^{\varepsilon}  \right\|_{\infty} \leq \frac{1}{1-2\,\varepsilon} \, \left\| \nabla u\left(\frac{\cdot}{1-2\varepsilon}\right)  \right\|_{p} \, \left\| \eta_{\varepsilon} \right\|_{p'} \leq
2\, \left\| \nabla u\left(\frac{\cdot}{1-2\varepsilon}\right)  \right\|_{p} \, \left\| \eta_{\varepsilon} \right\|_{p'},
\end{equation}
where $p'$ is the usual H{\"o}lder conjugate exponent. Using change of variables we obtain:
$$
 \left\| \eta_{\varepsilon} \right\|_{p'}^{p'} = 
 \int_{B_{\varepsilon}} \frac{1}{\varepsilon^{d\, p'}} \left|\eta\left(\frac{x}{\varepsilon} \right) \right|^{p'} \diff x = \varepsilon^{d\,(1-p')}  \int_{B} \left|\eta(x)\right|^{p'} \diff x = \varepsilon^{-p'\, \frac{d}{p}}  \| \eta \|_{p'}^{p'},
$$
so that $ \left\| \eta_{\varepsilon} \right\|_{p'} = \varepsilon^{-\frac{d}{p}}  \| \eta \|_{p'}$. Using change of variables again,
$$
\left\| \nabla u\left(\frac{\cdot}{1-2\varepsilon}\right)  \right\|_{p} \leq \left\| \nabla u  \right\|_{p}
$$
which is finite as $\mathcal{G}(u,B)<\infty$. Therefore, \eqref{eq:estimate_infinity2} boils down to
$$
\left\| \nabla u^{\varepsilon}  \right\|_{\infty} \leq D\, (5\varepsilon)^{-\frac{d}{p}},
$$
where $D:= 5^{\frac{d}{p}}\,\left\| \nabla u  \right\|_{p} \, \| \eta \|_{p'}$. Using Assumption \ref{ass:continuity_abstract} we obtain estimate \eqref{eq:estimate_for_infimum1}. The rest of the proof is exactly the same.
\end{proof}

\section{Proof of Theorem \ref{thm:main_theorem} in the general case}\label{sect:main_result_general}
\noindent In this section we generalize construction from Section \ref{sect:special_case} to prove Theorem \ref{thm:main_theorem} in the general case. 

\subsection{Second convex conjugate function}
For general $N$-function $\psi$ satisfying Assumption \ref{ass:Nfunction}, Lemma \ref{lem:existence_of_minimizer} is not necessarily true. Therefore, to control mollifications, we need a different method to approximate $\psi(x,\xi)$ with a function depending only on $\xi$. The construction below is somehow standard and has appeared in many works before, see \cite{chlebicka2018well,chlebicka2019parabolic}.\\

\noindent We start more generally. Let $f: \R \to \R$. We define convex conjugate $f^*:\R \to \R \cup \{+\infty\}$ of $f$ as
$$
f^*(\eta) = \sup_{\xi \in \R} \left(\xi \cdot \eta - f(\xi)\right).
$$
Moreover, the second convex conjugate of $f^{**}$ is defined as
$$
f^{**}(\xi) = \sup_{\eta \in \R} \left(\xi \cdot \eta - f^*(\eta)\right).
$$
We now list some basic properties of the convex conjugates cf. \cite[Propositions 2.21, 2.28]{MR3821514}.
\begin{lem}\label{lem:conv_conj_prop}
Let $f, g: \R \to \R$. Then, the following holds true:
\begin{enumerate}[label=(F\arabic*)]
    \item\label{lem:conv_conj_prop1} $f^*$ and $f^{**}$ are convex functions,
    \item\label{lem:conv_conj_prop2} if $f \leq g$ on $I$, then $g^* \leq f^*$ on $\R$, 
    \item\label{lem:conv_conj_prop3} if $f \leq g$ on $I$, then $f^{**} \leq g^{**}$ on $\R$,
    \item\label{lem:conv_conj_prop4} if $f$ is convex then $f^{**} = f$ on $\R$.
    \item\label{lem:conv_conj_prop5} $f^{**}$ is the gratest convex minorant of $f$.
\end{enumerate}
\end{lem}

\noindent Now, we apply these notions to $N$-functions. Given $N$-function $\psi(x,\xi)$ satisfying Assumption~\ref{ass:Nfunction}, we extend it by $0$ for $\xi < 0$ (hence this extension is surely convex), we consider a ball $B_{\gamma}(x)$ such that $\overline{B_{\gamma}(x)} \cap \overline{\Omega}$ is nonempty and we define
\begin{equation}\label{eq:def_2nd_convex_conj}
\psi_{x,\,\gamma}(\xi): \R \to \R, \qquad \qquad \psi_{x,\,\gamma}(\xi) :=  \inf_{y \in \overline{B_{\gamma}(x)} \cap \overline{\Omega}} \psi(y,\xi).
\end{equation}

\begin{lem}\label{lem:conjugates_estimate_from_both_sides}
Let $\psi$ be as in Assumptions \ref{ass:Nfunction}, \ref{ass:continuity_abstract} and $\psi_{x, \, \gamma}$ be as in \eqref{eq:def_2nd_convex_conj}. 
\begin{enumerate}[label=(G\arabic*)]
    \item\label{convex_conj_estimate_item1} Let $\mathcal{D}>1$. Then, there are constants $\mathcal{M} = \mathcal{M}(p,q,\mathcal{D})$, $\mathcal{N} = \mathcal{N}(p,q,\mathcal{D})$ such that
\begin{equation}\label{eq:estimate_with_2nd_conj}
\psi(y,\xi) \leq \mathcal{M} \, \psi_{x,\, \gamma}^{**}(\xi) + \mathcal{N}
\end{equation}
for all balls $B_{\gamma}(x)$, all $y \in \overline{B_{\gamma}(x)} \cap \overline{\Omega}$, all $\xi$ such that $\xi \leq \mathcal{D}\, \gamma^{-\min\left(1,\, \frac{d}{p}\right)}$ and all $\gamma \in \left(0,\frac{1}{2}\right)$.
    \item\label{convex_conj_estimate_item2} It holds $0 \leq \psi_{x,\, \gamma}^{**}(\xi) \leq \psi(y, \xi)$ for all balls $B_{\gamma}(x)$, all $y \in \overline{B_{\gamma}(x)} \cap \overline{\Omega}$ and all $\xi \in \R$.
\end{enumerate}
\end{lem}

\noindent One could try to prove Lemma \ref{lem:conjugates_estimate_from_both_sides} by applying property \ref{lem:conv_conj_prop3} from Lemma \ref{lem:conv_conj_prop} to the estimate appearing in Assumption \ref{ass:continuity_abstract}. However, this estimate is valid only on some bounded interval rather than the whole real line. The correct argument is presented in \cite{chlebicka2019book} but since it contains some imperfections, we present it below.

\begin{proof}[Proof of Lemma \ref{lem:conjugates_estimate_from_both_sides}]
The proof of \ref{convex_conj_estimate_item2} follows easily from \ref{lem:conv_conj_prop3} and \ref{lem:conv_conj_prop4} stated in Lemma~\ref{lem:conv_conj_prop}. For~\ref{convex_conj_estimate_item1} we split the proof into several steps. Recall that a convex function has a supporting line so that for all $\eta \in \R$, there exists supporting line $h_{\eta}$ such that $\psi_{x,\gamma}^{**}(\xi) \geq h_{\eta}(\xi)$ and $\psi_{x,\gamma}^{**}(\eta) = h_{\eta}(\eta)$.\\

\noindent \underline{Step 1.} The map $\R \ni \xi \mapsto \psi_{x,\,\gamma}(\xi)$ is locally Lipschitz continuous.

\begin{proof} Fix $y \in B_{\gamma}(x)$ and interval $[-R,R] \subset \R$. The map $\R \ni \xi \mapsto \psi(y,\xi)$ is convex so its difference quotients are monotone. Hence, for $\xi_1, \xi_2 \in [-R,R]$ with $\xi_1 < \xi_2$ we have
$$
\frac{\psi(y,\xi_2) - \psi(y,-R - 1)}{\xi_2 -(- R - 1)} \leq \frac{\psi(y,\xi_2) - \psi(y,\xi_1)}{\xi_2 - \xi_1} \leq \frac{\psi(y,R+1) - \psi(y,\xi_1)}{R+1 - \xi_1}
$$
Since $|\psi(y,R + 1)| \leq C_2\,(1+(R+1))^q$ cf. Assumption \ref{ass:Nfunction} \ref{ass:Nf_pq}, the map $y \mapsto \psi(y,\xi)$ is Lipschitz continuous with constant $2\,C_2\,(1+(R+1))^q$. As this holds uniformly for all $y \in B_{\gamma}(x)$, the conclusion follows.
\end{proof}

\noindent \underline{Step 2.} For $\xi \leq 0$ we have $\psi_{x,\,\gamma}^{**}(\xi) = 0$ and for $\xi > 0$ we have $\psi_{x,\,\gamma}^{**}(\xi) > 0$. In particular, estimate \eqref{eq:estimate_with_2nd_conj} is satisfied for $\xi \leq 0$. 

\begin{proof}
Consider extension of function $m_{\psi}(\xi)$  by 0 for $\xi \leq 0$. Then it follows from from~\ref{ass:lowerbound} in  Assumption~\ref{ass:Nfunction} that  
\begin{equation}\label{eq:lower_bound_2nd_conj_part_case}
m_{\psi}(\xi) \leq \psi_{x,\,\gamma}^{**}(\xi).
\end{equation}
This proves that $\psi_{x,\,\gamma}^{**}(\xi) \geq 0$ and for $\xi > 0$ we have $\psi_{x,\,\gamma}^{**}(\xi) > 0$. Finally, as $\psi_{x,\,\gamma}^{**}(\xi) \leq \psi_{x,\,\gamma}(\xi)$, we deduce $\psi_{x,\,\gamma}^{**}(\xi) = 0$ for $\xi \leq 0$.
\end{proof}

\noindent \underline{Step 3.} Fix $\eta$ such that $0 \leq \eta \leq \mathcal{D}\, \gamma^{-\min\left(1,\, \frac{d}{p}\right)}$ and assume that $h_{\eta}(\xi) = \psi_{x,\,\gamma}^{**}(\xi)$ only for $\xi = \eta$. Then, $\psi_{x,\,\gamma}^{**}(\eta) = \psi_{x,\,\gamma}(\eta)$ and estimate \eqref{eq:estimate_with_2nd_conj} is satisfied for $\xi = \eta$. 

\begin{proof}
Suppose that $\psi_{x,\,\gamma}^{**}(\eta) < \psi_{x,\,\gamma}(\eta)$ (we always have $\psi_{x,\,\gamma}^{**}(\eta) \leq \psi_{x,\,\gamma}(\eta)$!). Using Lipschitz continuity from Step 1, we find two lines such that $\psi_{x,\,\gamma}$ is above them (see dotted lines in Fig. \ref{fig:convexity}). Hence, we observe that $\psi_{x,\,\gamma}^{**}$ is not the largest convex minorant of $\psi_{x,\,\gamma}$, see Fig. \ref{fig:convexity}. It follows that $\psi_{x,\,\gamma}^{**}(\eta) = \psi_{x,\,\gamma}(\eta)$ and estimate \eqref{eq:estimate_with_2nd_conj} follows directly from Assumption \ref{ass:continuity_abstract}.
\end{proof}

\noindent \underline{Step 4.} Fix $\eta$ such that $0 \leq \eta \leq \mathcal{D}\, \gamma^{-\min\left(1,\, \frac{d}{p}\right)}$ and assume that $h_{\eta}(\xi) = \psi_{x,\,\gamma}^{**}(\xi)$ for some interval $[a, b]$ containing $\eta$ (so that $h_{\eta}$ and $\psi_{x,\,\gamma}^{**}$ have joint line interval). Then, estimate \eqref{eq:estimate_with_2nd_conj} is satisfied for $\xi = \eta$.

\begin{proof}
First, from Step 2 we may assume that $a \geq 0$ and from \eqref{eq:lower_bound_2nd_conj_part_case} we deduce $b < \infty$ (as function $m_{\psi}$ is superlinear). Second, the reasoning from Step 3 shows that
$$
\psi_{x,\,\gamma}(a) = \psi_{x,\,\gamma}^{**}(a), \qquad \qquad
\psi_{x,\,\gamma}(b) = \psi_{x,\,\gamma}^{**}(b).
$$
Moreover, by the assumption, there exists $t \in [0,1]$ such that
$$
\psi_{x,\,\gamma}^{**}(\eta) = t\, \psi_{x,\,\gamma}^{**}(a) + (1-t)\, \psi_{x,\,\gamma}^{**}(b) = t\, \psi_{x,\,\gamma}(a) + (1-t)\, \psi_{x,\,\gamma}(b).
$$
By definition of $\psi_{x,\,\gamma}$, there exist sequences $\{x^a_n\}_{n \in \N}$, $\{x^b_n\}_{n \in \N} \subset B_\gamma(x)$ such that
\begin{equation}\label{eq:estimate_with_points_1/n}
\psi_{x,\,\gamma}^{**}(\eta) \geq t\, \psi(x^a_n,\,a) + (1-t)\, \psi(x^b_n,\,b) - \frac{1}{n}.
\end{equation}
With these at hand, we proceed to the final proof. By definition and convexity,
\begin{equation}\label{eq:convexity_psi_infimum}
\psi_{x, \gamma}(\eta) \leq 
\psi(x^b_n,\eta) \leq t\, \psi(x^b_n, a) + (1-t)\,\psi(x^b_n, b) 
\end{equation}
To apply \eqref{eq:estimate_with_points_1/n}, we have to replace $\psi(x^b_n, a)$ with $\psi(x^a_n, a)$. This can be done with Assumption \ref{ass:continuity_abstract}: we note that $|x_n^a - x^b_n| \leq 2\,\gamma$ so if we let $D := 2^{\min\left(1,\frac{d}{p}\right)} \, \mathcal{D}$ we have
$$
|\eta| \leq \mathcal{D}\, \gamma^{-\min\left(1,\, \frac{d}{p}\right)} = D \, \left(2\,\gamma\right)^{-\min\left(1,\, \frac{d}{p}\right)} 
$$
and Assumption \ref{ass:continuity_abstract} implies existence of constants $M(D), N(D)$ (we skip dependence of these constants on $p$ and $q$ as these exponents are fixed) such that
$$
\psi(x^b_n,\eta) \leq M(D)\, \psi(x^a_n, a) + N(D).
$$
It follows from \eqref{eq:convexity_psi_infimum} that
$$
\psi_{x, \gamma}(\eta) \leq t\, \left(M(D)\, \psi(x^a_n, a) + N(D) \right) + (1-t)\, \psi(x^b_n, b) 
$$
Letting $\widetilde{M}(D) := \max(M(D),1)$ and exploiting \eqref{eq:estimate_with_points_1/n} we have
$$
\psi_{x, \gamma}(\eta) \leq \widetilde{M}(D) \left( t\, \psi(x^a_n,\,a) + (1-t)\, \psi(x^b_n,\,b) \right)  + N(D) \leq \widetilde{M}(D) \, \psi^{**}_{x,\,\gamma}(\xi_0) + \frac{\widetilde{M}(D)}{n} + N(D).
$$
Sending $n \to \infty$ we deduce
$$
\psi_{x, \gamma}(\eta) \leq \widetilde{M}(D) \, \psi^{**}_{x,\,\gamma}(\eta) + N(D).
$$
Exploiting Assumption \ref{ass:continuity_abstract} once again, we obtain for all $y \in B_{\gamma}(x)$
$$
\psi(y,\eta) \leq M(\mathcal{D})\, \psi_{x, \gamma}(\eta) + N(\mathcal{D}) \leq M(\mathcal{D})\,\widetilde{M}({D}) \, \psi^{**}_{x,\,\gamma}(\eta) +  N(\mathcal{D}) + N({D}).
$$
The conclusion follows with $\mathcal{M}:= M(\mathcal{D})\,\widetilde{M}({D})$ and $\mathcal{N} = N(\mathcal{D}) + N({D})$.
\end{proof}

\noindent \underline{Step 5.} Cases considered in Steps 2-4 are the only possible ones.

\begin{proof}
Clearly, the tangent line $h_{\eta}$ touches the epigraph of $\psi_{x,\gamma}^{**}$ at least in one point. The case where it is touched exactly at one point was studied in Step 3 while the situation when it is touched along some interval $[a,b]$ was analyzed in Step 4. Now, suppose that there are $\eta < \eta_1 < \eta_2$ such that
$$
\psi_{x,\gamma}^{**}(\eta) = h_{\eta}(\eta), \qquad \psi_{x,\gamma}^{**}(\eta_2) = h_{\eta}(\eta_2), \qquad \psi_{x,\gamma}^{**}(\eta_1) > h_{\eta}(\eta_1).
$$
Then, $\psi_{x,\gamma}^{**}$ is not convex raising contradiction.
\end{proof}

\end{proof}


\begin{figure}
\begin{tikzpicture}[scale = 3]

\coordinate (0) at (0, 0);
\coordinate (l) at (-1,-0.5);
\coordinate (r) at (1, 0.5);

\draw (l) -- (r);
\draw[fill] (0,0) circle [radius = 0.03];
\node [below, black] at (0) {$\eta$};
\node [right, black] at (r) {$h_{\eta}(\xi)$};

\draw[black, very thick] plot [smooth] coordinates {(-1,0.7) (0) (1,1.2)};
\node [below, black] at (0.99,0.94) {$\psi_{x,\varepsilon}^{**}(\xi)$};

\draw[lightgray, very thick] plot [smooth] coordinates {(-1,0.7) (-0.3, 0.13) (0, 0.7) (0.3, 0.28) (1,1.2)};
\node [gray!70!black, left] at (0.88,1) {$\psi_{x,\varepsilon}(\xi)$};

\draw[dotted, thick] (0, 0.7) -- (0.22,0.17);
\draw[dotted, thick] (0, 0.7) -- (-0.19,0.03);
\draw[dashed, thick] (0.22,0.17) -- (-0.19,0.02);


\end{tikzpicture}

\caption{We assume that there is $\eta > 0$ such that functions $\psi_{x,\,\gamma}^{**}$ (black line) and $\psi_{x,\,\gamma}$ (grey line) satisfy $\psi_{x,\,\gamma}^{**}(\eta) < \psi_{x,\,\gamma}(\eta)$ and tangent line $h_{\eta}$ touches $\psi_{x,\,\gamma}^{**}$ only at $\eta$. As $\psi_{x,\,\gamma}$ is Lipschitz continuous, we can estimate it from below (dotted lines). Then, the function obtained by combining $\psi_{x,\,\gamma}^{**}$ and the dashed line is convex. It lies below $\psi_{x,\,\gamma}$ and above $\psi_{x,\,\gamma}^{**}$ raising contradiction with Lemma \ref{lem:conv_conj_prop} \ref{lem:conv_conj_prop5}.}
\centering
\label{fig:convexity}
\end{figure}
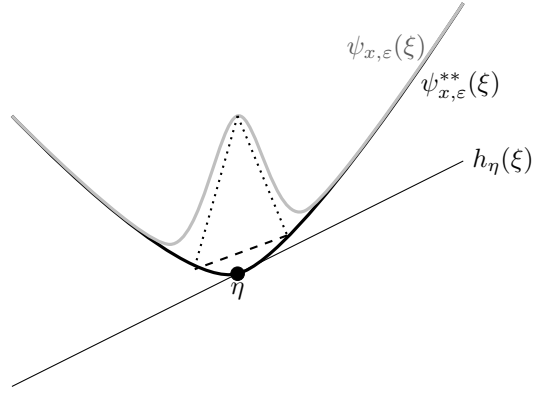

\subsection{Geometric issues}
As $\Omega$ is not a ball in general, we cannot define compactly supported approximation by retracting the function to the interior part of $\Omega$ as in Definition \ref{def:mol_in_sp}. However, one can still do that for star-shaped domains.

\begin{Def}
(1) A bounded domain $U \subset \R^d$ is said to be star-shaped with respect to $\overline{x}$ if every ray starting from $\overline{x}$ intersects with $\partial U$ at one and only one point. \\
(2) A bounded domain $U \subset \R^d$ is said to be star-shaped with respect to the ball $B_{\gamma}(x_0)$ if $U$ is star-shaped with respect to all $y \in B_{\gamma}(x_0)$.
\end{Def}

\noindent The following lemma shows that star-shaped domains can be uniformly shrinked which allows for defining compactly supported approximations.

\begin{lem}\label{lem:uniform_squeezing_star}
Let $U \subset \R^d$ be a star-shaped domain with respect to the ball $B_R$. Let $\kappa_{\varepsilon} = 1- \frac{4\,\varepsilon}{R}$. Then, $\dist(\kappa_{\varepsilon} \, U, \partial U) \geq 2\,\varepsilon$. In particular,
$$
\overline{\kappa_{\varepsilon} \, U + \varepsilon \, B } \subset U.
$$
More generally, if $U$ is star-shaped with respect to the ball $B_R(x_0)$,
$$
\overline{\kappa_{\varepsilon} \, (U-x_0) + \varepsilon \, B } \subset (U-x_0).
$$
\end{lem}
\begin{proof}
Let $b \in \partial U$ and let $c \in \partial (\kappa_{\varepsilon} U)$ such that $c$ lies on the interval $[0,b]$. Let $T \subset B_R$ be a sphere of radius $R$ perpendicular to the interval $[0,b]$ and let $S$ be the cone with base $T$ and apex $b$, see Figure \ref{fig:starshaed}. First, we have
$$
 \dist(\partial \Omega, c) \geq \dist(\partial S, c).
$$
Let $\alpha$ be a half of an apex angle of the cone $C$, see Figure \ref{fig:starshaed}. It follows that
$$
\dist(\partial S, c) = \sin(\alpha) \, |b-c| = \sin(\alpha) \, (1-\kappa_{\varepsilon}) \, |b| =  \sin(\alpha) \, (1-\kappa_{\varepsilon}) \, |b| \geq \sin(\alpha) \, \frac{4\,\varepsilon}{R} \, |b|.
$$
so it is sufficient to estimate $\sin(\alpha)$ from below. Using notation from Figure \ref{fig:starshaed}, the length of interval $[d,b]$ equals $\frac{b}{\cos(\alpha)}$. Therefore,
$$
\sin(\alpha) = \frac{R}{|b|/\cos(\alpha)} \implies \tan(\alpha) = \frac{R}{|b|}
$$
As $\sin^2(\alpha) = \frac{\tan^2(\alpha)}{1+\tan^2(\alpha)}$ we have
$$
\sin^2(\alpha) = \frac{R^2}{R^2 + b^2} \geq \frac{R^2}{4b^2} \implies \sin(\alpha) \geq \frac{R}{2|b|},
$$
where we used $R^2 \leq |b|^2 \leq 3\,|b|^2$. We conclude that $\dist(\partial U, c) \geq 2 \varepsilon$. As this argument can be repeated for all $c \in \partial (\kappa_{\varepsilon} U)$, we obtain $\dist(\partial U, \kappa_{\varepsilon} U) \geq 2\,\varepsilon$. The second statement follows from observation that the set $U-x_0$ is star-shaped with respect to the ball $B_R$.
\end{proof}


\begin{figure}
\begin{tikzpicture}[scale = 6]

\coordinate (0) at (0, 0);
\coordinate (c1) at (1.6, 0);
\coordinate (c2) at (1, 1);  
\coordinate (c3) at (0.3, 0.8);   
\coordinate (c4) at (-0.3, 0);
\coordinate (c5) at (0.4, -0.4);
\coordinate (x0) at (0.7, 0.2); 
\coordinate (c) at (0.87, 196/300);

\draw[black, ultra thick] plot [smooth cycle] coordinates {(c1) (c2) (c3) (c4) (c5)};

\draw[black] (x0) -- (c2);

\draw[fill] (0.7,0.2) circle [radius = 0.015];
\node [below left, black] at (x0) {$0$};

\draw[fill] (1,1) circle [radius = 0.015];
\node [above right, black] at (c2) {$b$};

\draw[fill] (0.87, 196/300) circle [radius = 0.015];
\node [right, black] at (c) {$c$};

\draw[black, dashed] (0.2,31/80) -- (1.2,1/80);

\draw plot [smooth, tension = 1.8] coordinates {(0.2,31/80) (0.6,-1/15) (1.2,1/80)}; 

\draw [dashed] plot [smooth, tension = 1.8] coordinates {(0.2,31/80) (0.79,0.4) (1.2,1/80)}; 

\draw[black]  (0.2,31/80) -- (c2);
\draw[black]  (1.2,1/80) -- (c2);

\coordinate (d1) at (0.2,31/80);
\coordinate (d2) at (1.2,1/80);
\coordinate (e1) at (5011/6675, 5401/6675); 
\coordinate (e2) at (189/445, 249/445); 

\pic [draw, -, "$\cdot$", angle eccentricity = 0.5] {angle = c--x0--d1};

\pic [draw, -, "$\cdot$", angle eccentricity = 0.5] {angle = c--e1--c2};

\pic [draw, -, "$\alpha$", angle eccentricity = 0.6, angle radius = 1 cm] {angle = d1--c2--c};

\draw[black, dotted, thick]  (e1) -- (c);
\draw[fill] (e1) circle [radius = 0.015];
\node [above left, black] at (e1) {$e$};
\draw[fill] (d1) circle [radius = 0.015];
\node [left, black] at (d1) {$d\,$};

\node [scale = 1.5] at (1.68,0) {$\partial U$};
\node [scale = 1.5] at (1.18,0.5) {$\partial S$};
\node [scale = 1.5] at (0.3,0) {$T$};

\end{tikzpicture}

\caption{Three dimensional adaptation of the construction performed in Lemma \ref{lem:uniform_squeezing_star}. Point $b$ belongs to the boundary of star-shaped domain $U$ while point $c = \kappa_{\varepsilon} b$ belongs to the boundary of the rescaled set $\kappa_{\varepsilon} U$. Sphere $T \subset B_R$ is perpendicular to the interval $[0,b]$ and since $U$ is star-shaped with respect to the ball $B_R$, the cone $S$ with base $T$ and apex $b$ lies inside $U$.}
\centering
\label{fig:starshaed}
\end{figure}
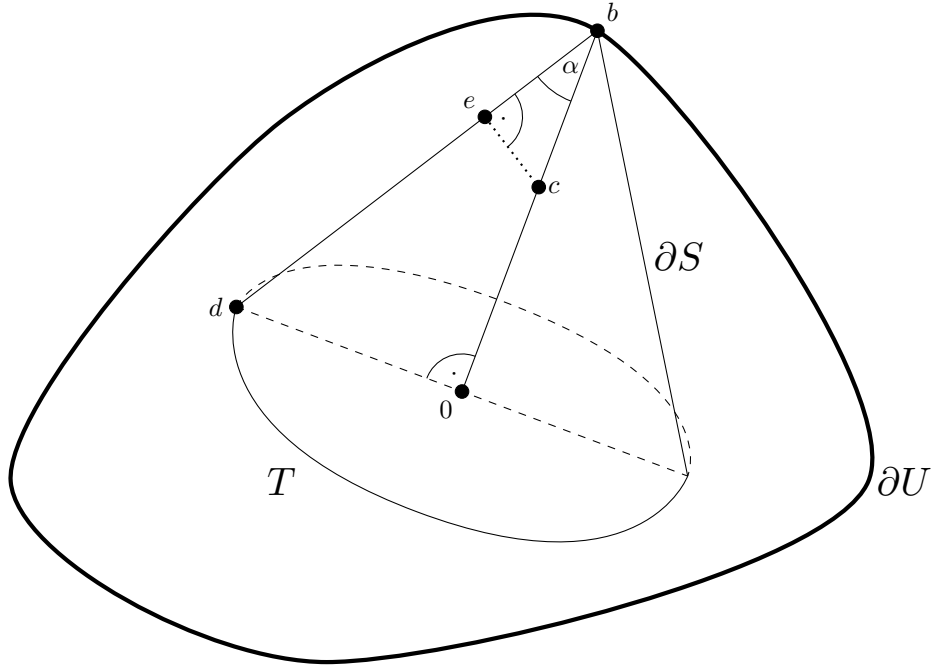


\noindent On star-shaped domain we can define mollification with squeezing as in Definition \ref{def:mol_in_sp}: 

\begin{Def}[Mollification with squeezing on star-shaped domain]\label{def:mol_in_star_shaped}
Let $U$ be a star-shaped domain with respect to the ball $B_R(x_0)$. Given $u \in W^{1,1}_0(U)$ we extend it with $0$ to $\R^d$ and define
$$
    \mathcal{S}^{\varepsilon}_U u(x) := \int_{\R^d} u\left(x_0 +  \frac{x-x_0 -y}{\kappa_{\varepsilon}} \right) \, \eta_{\varepsilon}(y) \diff y,
$$
where $\kappa_{\varepsilon} = 1 - \frac{4\,\varepsilon}{R}$.
\end{Def}

\noindent Reader may think about the case $x_0 = 0$ first.

\begin{lem}\label{lem:compact_support_smooth_generalOmega} Function $\mathcal{S}^{\varepsilon}_U u$ from Definition \ref{def:mol_in_star_shaped} belongs to $C_c^{\infty}(U)$.
\end{lem}
\begin{proof}
The smoothness is clear from standard properties of convolutions. Concerning compact support, we claim that $\mathcal{S}^{\varepsilon}_U u$ is supported in $x_0 + \overline{\kappa_{\varepsilon} \, (U-x_0) + \varepsilon \, B }$ which is a compact subset of $U$ due to Lemma \ref{lem:uniform_squeezing_star}. Indeed, let $x \notin x_0 + \overline{\kappa_{\varepsilon} \, (U-x_0) + \varepsilon \, B }$ and suppose that there is $y$ with $|y| \leq \varepsilon$ such that $x_0 + \frac{x-x_0 -y}{\kappa_{\varepsilon}} \in U$. Then, we can write
$$
x = x_0 +  \kappa_{\varepsilon} \, \left(x_0 + \frac{x - x_0 -y}{\kappa_{\varepsilon}} - x_0\right) + y
$$
so that $x \in x_0 + \overline{\kappa_{\varepsilon} \, (U-x_0) + \varepsilon \, B}$ raising contradiction. It follows that for $x \in x_0 + \overline{\kappa_{\varepsilon} \, (U-x_0) + \varepsilon \, B}$ we have either 
$$
x_0 + \frac{x-x_0 -y}{\kappa_{\varepsilon}} \in U \mbox{ and } |y| > \varepsilon 
\qquad 
\mbox{ or } 
\qquad 
x_0 + \frac{x-x_0 -y}{\kappa_{\varepsilon}} \notin U
$$
so that the integral $\int_{\R^d} u\left(x_0 +  \frac{x-x_0 -y}{\kappa_{\varepsilon}} \right) \, \eta_{\varepsilon}(y) \diff y = 0$.
\end{proof}

\noindent To move from star-shaped domains to Lipschitz ones we will use the following decomposition cf. \cite[Lemma 3.14]{MR2084891}.

\begin{lem}\label{lem:star-sh-dec}
Suppose that $\Omega \subset \R^d$ is a bounded Lipschitz domain. Then, there exist domains $\{U_i\}_{i=1,...,n}$ such that 
$$
\overline{\Omega} \subset \bigcup_{i=1}^n U_i.
$$
and $\Omega \cap U_i$ is star-shaped with respect to some ball $B_{R_i}(x_i)$.
\end{lem}

\subsection{Approximating sequence and proof of Theorem \ref{thm:main_theorem}}

\noindent We are in position to define the approximating sequence. Let $\Omega$ be a Lipschitz bounded domain. From Lemma \ref{lem:star-sh-dec} we obtain family of domains such that $\overline{\Omega} \subset \bigcup_{i=1}^n U_i$ where $\{\Omega \cap U_i\}_{i=1,...,n}$ are star-shaped domains with respect to balls $B_{R}(x_i)$ (without loss of generality, we may assume that the radii of the balls are the same by taking $R := \min_{i=1,...n} R_i$). In particular, $\{U_i\}_{i=1,...,n}$ is an open covering of $\overline{\Omega}$ so there exists partition of unity related to this covering: family of functions $\{\theta_i\}_{i=1,...,n}$ such that
$$
\theta_i \in C_c^{\infty}(U_i), \qquad 0 \leq \theta_i \leq 1, \qquad \sum_{i=1}^n \theta_i = 1 \mbox{ on } \overline{\Omega}.
$$
Given $u \in W^{1,1}_0(\Omega) \cap L^{\infty}(\Omega)$ we extend it with 0 as above and we set
\begin{equation}\label{eq:approx_seq_generalOmega}
\mathcal{S}^{\varepsilon} u :=  \sum_{i=1}^n \mathcal{S}^{\varepsilon}_{U_i}(u \, \theta_i) = \sum_{i=1}^n 
\int_{B_{\varepsilon}} ( u \, \theta_i) \left(x_i + \frac{x-x_i - y}{\kappa_{\varepsilon}}\right)  \, \eta_{\varepsilon}(y) \diff y
\end{equation}
where $\kappa_{\varepsilon} = 1 - \frac{4\,\varepsilon}{R}$. We note that since $u$ vanishes outside of $\Omega$, function $u\, \theta_i$ is supported in $\Omega \cap U_i$ which is star-shaped. \\

\noindent Before formulating the main result of this section, we will state and prove two technical lemmas concerning approximating sequence. 
\begin{lem}\label{lem:balls_uni_scaling}
Let $\kappa_{\varepsilon} = 1 - \frac{4\,\varepsilon}{R}$, $x \in \Omega$ and $|y| \leq \varepsilon$. Then, there exists a constant $C_{\Omega,R}$ such that for $\varepsilon \leq \frac{R}{8}$ we have $x_i + \frac{x-x_i - y}{\kappa_{\varepsilon}} \in B_{C_{\Omega,R} \varepsilon}(x)$.
\end{lem}
\begin{proof}
Note that for $\varepsilon \leq \frac{R}{8}$, we have $\frac{1}{\kappa_{\varepsilon}} \leq 2$. We compute
$$
\left|x_i + \frac{x-x_i - y}{\kappa_{\varepsilon}} - x \right| = 
\left|(x_i-x)\left(1-\frac{1}{\kappa_{\varepsilon}}\right) - \frac{y}{\kappa_{\varepsilon}} \right| \leq |x_i - x| \frac{1-\kappa_{\varepsilon}}{\kappa_{\varepsilon}} +  \frac{\varepsilon}{\kappa_{\varepsilon}} \leq 8\,\frac{|x_i - x|}{R}\, \varepsilon + 2\,\varepsilon.
$$
As $|x_i - x| \leq \mbox{diam}(\Omega)$ (the diameter of $\Omega$), we choose $C_{\Omega, R} := \frac{8 \, \text{diam}(\Omega)}{R} + 2$.
\end{proof}

\begin{lem}\label{lem:technical_convergence_of_mollification_gen}
Let $u \in W^{1,1}_0(\Omega)$ be such that $\mathcal{H}(u)<\infty$ and consider its extension to $\R^d$. Then,
\begin{enumerate}[label=(H\arabic*)]
    \item \label{item:conv_large_term_general1} $\psi\left(x_i + \frac{x-x_i}{\kappa_{\varepsilon}}, (\left|\nabla u\right| \, \theta_i)\left(x_i + \frac{x-x_i}{\kappa_{\varepsilon}}\right) \right) \to \psi(x,\left|\nabla u\right| \, \theta_i)$ in $L^1(\R^d)$,
    \item \label{item:conv_large_term_general2} $ \int_{B_{\varepsilon}}   \psi\left(x_i + \frac{x-x_i - y}{\kappa_{\varepsilon}},  (\left|\nabla u\right| \, \theta_i)\left(x_i + \frac{x-x_i - y}{\kappa_{\varepsilon}}\right) \right)  \eta_{\varepsilon}(y) \diff y \to \varphi\left(x,  \left|\nabla u\right| \, \theta_i \right)$ in $L^1(\R^d)$.
\end{enumerate}
\end{lem}
\begin{proof}
Concerning \ref{item:conv_large_term_general1}, we note that the convergence holds in the pointwise sense. Moreover, the considered sequence is supported on $\Omega \cap U_i$. Therefore, to establish convergence in $L^1(\R^d)$, it is sufficient to prove equiintegrability of the sequence $\left\{\psi\left(x_i + \frac{x-x_i}{\kappa_{\varepsilon}}, (\left|\nabla u\right| \, \theta_i)\left(x_i + \frac{x-x_i}{\kappa_{\varepsilon}}\right) \right)\right\}_{\varepsilon}$ and apply Vitali convergence theorem. To this end, we need to prove  
$$
\forall_{\eta>0}\, \exists_{\delta>0} \, \forall_{A\subset \Omega \cap U_i, |A|\leq \delta} \, \qquad \int_{A} \psi\left(x_i + \frac{x-x_i}{\kappa_{\varepsilon}},  (\left|\nabla u\right| \, \theta_i)\left(x_i + \frac{x-x_i}{\kappa_{\varepsilon}}\right) \right) \diff x \leq \eta.
$$
We fix $\eta$ and arbitrary $A \subset \Omega \cap U_i$. Using convexity,
$$
\psi\left(x_i + \frac{x-x_i}{\kappa_{\varepsilon}}, (\left|\nabla u\right| \, \theta_i)\left(x_i + \frac{x-x_i}{\kappa_{\varepsilon}}\right) \right) \leq \psi\left(x_i + \frac{x-x_i}{\kappa_{\varepsilon}}, \left|\nabla u\right|\left(x_i + \frac{x-x_i}{\kappa_{\varepsilon}}\right) \right)
$$
as $0 \leq \theta_i \leq 1$. Second, using change of variables we have  
\begin{multline*}
\int_{A} \psi\left(x_i + \frac{x-x_i}{\kappa_{\varepsilon}}, \left|\nabla u\right|\left(x_i + \frac{x-x_i}{\kappa_{\varepsilon}}\right) \right) \diff x = (\kappa_{\varepsilon})^d \int_{\widetilde{A}} \psi(x,\left|\nabla u\right|(x)) \diff x \leq \int_{\widetilde{A}} \psi(x,\left|\nabla u\right|(x)) \diff x,
\end{multline*}
where $\widetilde{A}$ is a set obtained from $A$ after the performed change of variables. Note that measures of these sets satisfy
$
|\widetilde{A}|\leq \frac{1}{\kappa_{\varepsilon}^d} |A| \leq 2^d \, |A|.
$
Having this in mind, we let 
$$
\omega(\tau) := \sup_{C \subset \R^d: |C|\leq \tau} \int_C \psi(x,\left|\nabla u\right|(x)) \diff x.
$$
Function $\omega(\tau)$ is a non-decreasing function, continuous at 0 because $\mathcal{H}(u, \Omega) < \infty$. Therefore, we may find $\tau$ such that $\omega(\tau) \leq  \eta$. Then, we choose $\delta = 2^{-d} \, \tau$ to concude the proof of \ref{item:conv_large_term_general1}. Finally, \ref{item:conv_large_term_general2} follows from Young's convolutional inequality and \ref{item:conv_large_term_general1}.
\end{proof}

\begin{thm}[Theorem \ref{thm:main_theorem} in the general case]\label{res:approx_theorem_GENERAL}
Let $u \in W^{1,\psi}_0(\Omega) \cap L^{\infty}(\Omega)$ where $\psi$ satisfies Assumptions \ref{ass:Nfunction} and \ref{ass:continuity_abstract}. Let $\mathcal{H}$ be given with \eqref{eq:general_functional}. Suppose that 
$$
q \leq p + {\alpha} \max\left( 1, \frac{p}{d}\right).
$$
Consider sequence $\left\{\mathcal{S}^{\varepsilon} u\right\}_{\varepsilon > 0}$ as in \eqref{eq:approx_seq_generalOmega} with 
$\varepsilon \leq \frac{R}{8}$. Then,
\begin{enumerate}[label=(I\arabic*)]
\item \label{item:main_thm_gen_item1} $\mathcal{S}^{\varepsilon} u \in C_c^{\infty}(\Omega)$,
\item \label{item:main_thm_gen_item2} $\mathcal{H}\left(\mathcal{S}^{\varepsilon} u, \Omega\right) \to \mathcal{H}(u, \Omega)$ as $\varepsilon \to 0$,
\item \label{item:main_thm_gen_item3} $\mathcal{S}^{\varepsilon} u \to u$ in $W^{1,\psi}(\Omega)$ as $\varepsilon \to 0$, 
\item \label{item:main_thm_gen_item4} space $C_c^{\infty}(\Omega)$ is dense in $W^{1,\psi}_0(\Omega) $ and Lavrentiev phenomenon does not occur, i.e. for all boundary data $u_0 \in W^{1,q}(\Omega)$:
\begin{equation*}
    \inf_{u \in u_0 + W_0^{1,p}(\Omega)} \mathcal{H}(u, \Omega) = \inf_{u \in u_0 + W_0^{1,q}(\Omega)} \mathcal{H}(u, \Omega) = \inf_{u \in u_0 + C_c^{\infty}(\Omega)} \mathcal{H}(u, \Omega).
\end{equation*}
\end{enumerate}
\end{thm}

\begin{proof}
The first property follows from Lemma \ref{lem:compact_support_smooth_generalOmega}. To prove convergence, we note that 
$$
\mathcal{H}\left(\mathcal{S}^{\varepsilon} u, \Omega\right) =
\int_{\Omega} \psi(x, \nabla \mathcal{S}^{\varepsilon} u(x) ) \diff x.
$$
To take mollification out of the function $\psi$ we want to use Jensen's inequality and Lemma \ref{lem:conjugates_estimate_from_both_sides}. The latter requires estimate on $\left\|\nabla \mathcal{S}^{\varepsilon} u \right\|_{\infty}$ \\

\noindent \underline{\textbf{Case 1: $p \leq d$.}} In this case we have $q \leq p + \alpha$. Using Young's convolution inequality we obtain:
\begin{equation}\label{eq:estimate_infinity_generalcase}
\left\|\nabla \mathcal{S}^{\varepsilon} u \right\|_{\infty}  \leq \sum_{i=1}^n \left\| u\, \theta_i \right\|_{\infty} \, \left\| \nabla \eta_{\varepsilon} \right\|_{1} \leq \sum_{i=1}^n \left\| u \right\|_{\infty} \, \left\| \nabla \eta_{\varepsilon} \right\|_{1} \leq \mathcal{D}\, (C_{\Omega, R}\,\varepsilon)^{-1},
\end{equation}
where we choose $\mathcal{D}:= n\, \|u\|_{\infty} \,\left\| \nabla \eta \right\|_{1} \, C_{\Omega, R}$ and $C_{\Omega, R}$ is a constant from Lemma \ref{lem:balls_uni_scaling}. Let $x \in \Omega$. Applying Lemma \ref{lem:conjugates_estimate_from_both_sides} with $\gamma = C_{\Omega, R}\,\varepsilon$ we obtain constants $\mathcal{M}$, $\mathcal{N}$ such that 
\begin{equation}\label{eq:estimate_split_xi}
\psi\left(x, \left|\nabla \mathcal{S}^{\varepsilon} u \right|\right) \leq \mathcal{M} \, \psi_{x,\,\gamma}^{**}\left(\left|\nabla \mathcal{S}^{\varepsilon} u\right|\right) + \mathcal{N}
\end{equation}
where function $\psi_{x,\,\gamma}^{**}$ is the second convex conjugate of the function defined in \eqref{eq:def_2nd_convex_conj}.
Now, we want to estimate $\psi_{x,\,\gamma}^{**}\left(\left|\nabla \mathcal{S}^{\varepsilon} u\right|\right)$. Due to its convexity, Jensen's inequality implies
\begin{align*}
\psi_{x,\,\gamma}^{**}\left(\,\left|\nabla \mathcal{S}^{\varepsilon} u\right|\right) &= \psi_{x,\,\gamma}^{**}\left(\left| \sum_{i=1}^n \int_{B_{\varepsilon}} \nabla_x ( u \, \theta_i) \left(x_i + \frac{x-x_i - y}{\kappa_{\varepsilon}}\right)  \, \eta_{\varepsilon}(y) \diff y \right|\,\right) \leq \\
&\leq
\int_{B_{\varepsilon}} \psi_{x,\,\gamma}^{**}\left( \sum_{i=1}^n \left| \nabla_x (u \, \theta_i)\left(x_i + \frac{x-x_i - y}{\kappa_{\varepsilon}}\right) \right|\,\right)  \, \eta_{\varepsilon}(y) \diff y \\
&\leq 
\frac{1}{2}\,\int_{B_{\varepsilon}} \psi_{x,\,\gamma}^{**}\left( \frac{2}{\kappa_{\varepsilon}} \,\sum_{i=1}^n \left| (\nabla u \, \theta_i)\left(x_i + \frac{x-x_i - y}{\kappa_{\varepsilon}}\right) \right|\, \right)  \, \eta_{\varepsilon}(y) \diff y \\
&\phantom{\leq \,} + \frac{1}{2}\,\int_{B_{\varepsilon}} \psi_{x,\,\gamma}^{**}\left( \frac{2}{\kappa_{\varepsilon}} \,\sum_{i=1}^n \left|(u \, \nabla \theta_i)\left(x_i + \frac{x-x_i - y}{\kappa_{\varepsilon}}\right) \right|\,\right)  \, \eta_{\varepsilon}(y) \diff y =: X+Y.
\end{align*}
Concerning term $Y$, using upper bound from Lemma \ref{lem:conjugates_estimate_from_both_sides} and $q$-growth cf. Assumption \ref{ass:Nfunction} \ref{ass:Nf_pq} we can estimate:
\begin{multline*}
 \psi_{x,\,\gamma}^{**}\left( \frac{2}{\kappa_{\varepsilon}} \,\left|\sum_{i=1}^n (u \, \nabla \theta_i)\left(x_i + \frac{x-x_i - y}{\kappa_{\varepsilon}}\right) \right|\, \right) \leq \\
 \leq C_2 \, \left(1+\left|  \frac{2\,n}{\kappa_{\varepsilon}}\, \|u\|_{\infty} \, \sup_{i=1,...,n} \|\nabla \theta_i\|_{\infty} \right|^q\right) \leq
 C_2 \, \left(1+\left|  4\,n\, \|u\|_{\infty} \, \sup_{i=1,...,n} \|\nabla \theta_i\|_{\infty} \right|^q\right),
\end{multline*}
where we used $\frac{1}{\kappa_{\varepsilon}} \leq 2$ for $\varepsilon \leq \frac{R}{8}$ in the second inequality. Therefore, using $\int_{B_{\varepsilon}} \eta_{\varepsilon}(y) \diff y = 1$
$$
Y \leq M\,C_2 \, \left(1+\left| 4\,n\, \|u\|_{\infty} \, \sup_{i=1,...,n} \|\nabla \theta_i\|_{\infty} \right|^q\right) := \|Y\|_{\infty}  < \infty,
$$
Concerning term $X$, we use Jensen's inequality again:
$$
X \leq \frac{M}{2\,n}\,\sum_{i=1}^n\int_{B_{\varepsilon}} \psi_{x,\,\gamma}^{**}\left( \frac{2\,n}{\kappa_{\varepsilon}} \left|  (\nabla u \, \theta_i)\left(x_i + \frac{x-x_i - y}{\kappa_{\varepsilon}}\right) \right|\, \right)  \, \eta_{\varepsilon}(y) \diff y
$$
so that we can study each summand independently. If $x_i + \frac{x-x_i - y}{\kappa_{\varepsilon}}$ does not belong to $\overline{\Omega}$ then $\psi_{x,\,\gamma}^{**}\left( \frac{2n}{\kappa_{\varepsilon}} \left|(\nabla u \, \theta_i)\left(x_i + \frac{x-x_i - y}{\kappa_{\varepsilon}}\right) \right| \right)= 0$ because $\nabla u$ vanishes at this point cf. Lemma \ref{lem:conjugates_estimate_from_both_sides} \ref{convex_conj_estimate_item2}. Otherwise, $x_i + \frac{x-x_i - y}{\kappa_{\varepsilon}} \in \overline{\Omega} \cap \overline{B_{C_{R,\Omega}\, \varepsilon}(x)}$ cf. Lemma \ref{lem:balls_uni_scaling} so that
\begin{equation*}
\psi_{x,\,\gamma}^{**}\left( \frac{2\,n}{\kappa_{\varepsilon}} \left| (\nabla u \, \theta_i)\left(x_i + \frac{x-x_i - y}{\kappa_{\varepsilon}}\right) \right|\, \right) \leq 
\psi\left(x_i + \frac{x-x_i - y}{\kappa_{\varepsilon}}, \frac{2\,n}{\kappa_{\varepsilon}}\, \left| (\nabla u \, \theta_i)\left(x_i + \frac{x-x_i - y}{\kappa_{\varepsilon}}\right) \right|\, \right)
\end{equation*}
due to Lemma \ref{lem:conjugates_estimate_from_both_sides}. By virtue of \eqref{eq:comp_dc} and Assumption \ref{ass:Nfunction} \ref{ass:Nf_delta2}, we have
\begin{multline*}
\psi\left(x_i + \frac{x-x_i - y}{\kappa_{\varepsilon}}, \frac{2\,n}{\kappa_{\varepsilon}}\, \left| (\nabla u \, \theta_i)\left(x_i + \frac{x-x_i - y}{\kappa_{\varepsilon}}\right) \right|\, \right) 
\leq \\ \leq
C_4^k\, \psi\left(x_i + \frac{x-x_i - y}{\kappa_{\varepsilon}}, \left|(\nabla u \, \theta_i)\left(x_i + \frac{x-x_i - y}{\kappa_{\varepsilon}}\right) \right|\, \right),
\end{multline*}
where $k$ is the smallest natural number such that $\frac{2\,n}{\kappa_{\varepsilon}} \leq 2^k$. Using \eqref{eq:estimate_split_xi} we conclude
\begin{equation}\label{eq:crucial_estimate_two_UNI_INT_general_domain}
\begin{split}
&\psi\left(x,\left|\nabla \mathcal{S}^{\varepsilon} u\right|\right) \leq \mathcal{N} + \|Y\|_{\infty} + \\
& \qquad  + C_4^k\,
\frac{\mathcal{M}}{2\,n}\, \sum_{i=1}^n \int_{B_{\varepsilon}}   \psi\left(x_i + \frac{x-x_i - y}{\kappa_{\varepsilon}}, \left| (\nabla u \, \theta_i)\left(x_i + \frac{x-x_i - y}{\kappa_{\varepsilon}}\right)\right|\, \right)  \eta_{\varepsilon}(y) \diff y.
\end{split}
\end{equation}
Now, we observe that $\psi\left(x,\left|\nabla \mathcal{S}^{\varepsilon} u\right|\right)$ converges a.e. to $\psi(x, \left|\nabla u\right| )$. Moreover, the first term on (RHS) of \eqref{eq:crucial_estimate_two_UNI_INT_general_domain} is convergent in $L^1(\Omega)$ cf. Lemma \ref{lem:technical_convergence_of_mollification_gen} \ref{item:conv_large_term_general2}. Therefore, Corollary \ref{cor:vitali} (Vitali convergence theorem) implies
$$
\psi(x, \left|\nabla \mathcal{S}^{\varepsilon} u\right|) \to \psi(x, \left|\nabla u\right|) \qquad \mbox{ in } L^1(\Omega)  \mbox{ as } \varepsilon \to 0.
$$
Thanks to triangle inequality we obtain \ref{item:main_thm_gen_item2}. Now, \ref{item:main_thm_gen_item3} follows from Lemma \ref{lem:conv_mus_orlicz_spaces_Delta2} \ref{item:mus-or:4} while property \ref{item:main_thm_gen_item4} follows from Lemma \ref{lem:density_implies_lavr}. \\

\noindent \underline{\textbf{Case 2: $p > d$.}} In this case we have $q \leq p +\alpha\, \frac{p}{d}$. In this situation, instead of \eqref{eq:estimate_infinity_generalcase}, we compute using change of variables
\begin{equation}\label{eq:estimate_infinity2_GENERAL}
\left\| \nabla \mathcal{S}^{\varepsilon} u \right\|_{\infty} \leq \frac{1}{\kappa_{\varepsilon}}\,\sum_{i=1}^n \left\| \nabla (u\, \theta_i) \right\|_{p} \, \left\| \eta_{\varepsilon} \right\|_{p'} \leq 2 \,\sum_{i=1}^n \left\| \nabla (u\, \theta_i) \right\|_{p} \, \left\| \eta_{\varepsilon} \right\|_{p'},
\end{equation}
where $p'$ is the usual H{\"o}lder conjugate exponent. Using change of variables we obtain:
$$
 \left\| \eta_{\varepsilon} \right\|_{p'}^{p'} = 
 \int_{B_{\varepsilon}} \frac{1}{\varepsilon^{d\, p'}} \left|\eta\left(\frac{x}{\varepsilon} \right) \right|^{p'} \diff x = \varepsilon^{d\,(1-p')}  \int_{B} \left|\eta(x)\right|^{p'} \diff x = \varepsilon^{-p'\, \frac{d}{p}}  \| \eta \|_{p'}^{p'},
$$
so that $ \left\| \eta_{\varepsilon} \right\|_{p'} = \varepsilon^{-\frac{d}{p}}  \| \eta \|_{p'}$. Concerning the term with function $u$,
$$
 \left\| \nabla (u\, \theta_i) \right\|_{p} \leq \left\| \nabla u \right\|_{p} \,  \left\| \theta_i \right\|_{\infty} + \left\| u \right\|_{p} \,  \left\| \nabla \theta_i \right\|_{\infty}
 $$
which is finite as $\mathcal{H}(u, \Omega)<\infty$ and $u \in L^{\infty}(\Omega)$. Therefore, \eqref{eq:estimate_infinity2} boils down to
$$
\left\| \nabla \mathcal{S}^{\varepsilon} u  \right\|_{\infty} \leq D\, \left( C_{\Omega,R}\,\varepsilon\right)^{-\frac{d}{p}}, \qquad D:= 2\, C_{\Omega,R}^{\frac{d}{p}}\,\| \eta \|_{p'} \,\sum_{i=1}^n \left(\left\| \nabla u \right\|_{p} \,  \left\| \theta_i \right\|_{\infty} + \left\| u \right\|_{p} \,  \left\| \nabla \theta_i  \right\|_{\infty}\right).
$$
Now, we can apply Lemma \ref{lem:conjugates_estimate_from_both_sides} \ref{convex_conj_estimate_item1} to obtain estimate \eqref{eq:estimate_split_xi}. The rest of the proof is exactly the same.
\end{proof}

\section{Extension of Theorem \ref{thm:main_theorem} to vector-valued maps}\label{sect:var-exp}

\noindent Many authors consider variational problems with vector-valued functions. However, in our work functionals depend only on the length of the gradient so there is almost no difficulty in extending our result to the vector case setting. In this section, we write ${\bf u}=(u^1, ..., u^n)$ for the map $u: \Omega \to \R^n$. For simplicity, we use the same notation for spaces of vector-valued functions as for spaces of scalar-valued ones.\\

\noindent The main point that needs explanation is a generalisation of Lemma~\ref{lem:density_of_bounded_in_musielak}, where we applied truncation to approximate functions from $W^{1,\psi}(\Omega)$ by functions from $W^{1,\psi}(\Omega) \cap L^{\infty}(\Omega)$. 

\begin{lem}\label{lem:component_convergence}
Let $u: \Omega \to \R^n$, ${\bf u}=(u^1, \ldots, u^n)$ where $n \in \N$. Suppose that $u \in W^{1,\psi}(\Omega)$. Then, $u^i \in W^{1,\psi}(\Omega)$. Moreover, suppose that for each $i = 1, \ldots, n$ we have $u^i_k \to u^i$ in $W^{1,\psi}(\Omega)$. Let ${\bf u_k}:= (u^1_k, \ldots, u^n_k)$. Then, ${\bf u_k} \to {\bf u}$ in $W^{1,\psi}(\Omega)$.
\end{lem}
\begin{proof}
We observe that if we interpret $|\nabla {\bf u}|$ component-wisely, we have $|\nabla u^i| \leq |\nabla {\bf u}|$. By convexity of $\xi \mapsto \psi(x,\xi)$, we have
$$
0 \leq \psi(x,|\nabla {u^i}|) \leq \psi(x,|\nabla {\bf u}|).
$$
To see the second statement, we note
$$
|\nabla {\bf u} - \nabla {\bf u_k}| \leq \sum_{i=1}^n |\nabla {\bf u^i} - \nabla {\bf u_k^i}|,
$$
so that by convexity of the mapping $\xi \mapsto \psi(x,\xi)$
$$
0 \leq \psi(x,|\nabla {\bf u} - \nabla {\bf u_k}|) \leq \frac{1}{n} \sum_{i=1}^n\psi\left(x, n\, |\nabla { u^i} - \nabla { u_k^i}|\right).
$$
Using Lemma~\ref{lem:conv_mus_orlicz_spaces_Delta2} \ref{item:mus-or:2} we conclude that ${\bf u_k} \to {\bf u}$ in $W^{1,\psi}(\Omega)$.
\end{proof}

\begin{thm}\label{thm:main_theorem_vectorial_case}
Suppose that $p \leq q + \alpha \max\left(1, \frac{p}{d} \right)$. Let $\mathcal{H}$ be a functional defined by \eqref{eq:general_functional} with $\psi$ satisfying Assumptions \ref{ass:Nfunction} and \ref{ass:continuity_abstract}. Then, for all ${\bf u_0} \in W^{1,q}(\Omega)$ we have
$$
    \inf_{{\bf u} \in {\bf u_0} + W_0^{1,p}(\Omega)} \mathcal{H}({\bf u},\Omega) = \inf_{{\bf u} \in {\bf u_0}  + W_0^{1,q}(\Omega)} \mathcal{H}({\bf u},\Omega) = \inf_{{\bf u} \in {\bf u_0} + C_c^{\infty}(\Omega)} \mathcal{H}({\bf u},\Omega).
$$
Moreover, space $C_c^{\infty}(\Omega)$ is dense in the Musielak--Orlicz--Sobolev space $W^{1,\psi}_0(\Omega)$.
\end{thm}

\begin{proof}
We first prove that $C_c^{\infty}(\Omega)$ is dense in the Musielak--Orlicz--Sobolev space $W^{1,\psi}_0(\Omega)$. This follows from the following facts.
\begin{itemize}
\item $W^{1,\psi}_0(\Omega) \cap L^{\infty}(\Omega)$ is dense in $W^{1,\psi}_0(\Omega)$. Indeed, let ${\bf u} \in W^{1,\psi}_0(\Omega)$ and ${\bf u_k} := (T_k(u^1), ..., T_k(u^n))$. It follows from Lemmas~\ref{lem:density_of_bounded_in_musielak} and \ref{lem:component_convergence} that ${\bf u_k} \to {\bf u}$ in $W^{1,\psi}(\Omega)$.
\item $C_c^{\infty}(\Omega)$ is dense in $W^{1,\psi}_0(\Omega) \cap L^{\infty}(\Omega)$. Indeed, let ${\bf u} \in W^{1,\psi}_0(\Omega) \cap L^{\infty}(\Omega)$. Then, each $u^i \in W^{1,\psi}_0(\Omega) \cap L^{\infty}(\Omega)$. By Theorem \ref{thm:main_theorem}, we have a sequence $\{u^i_k\}_{k \in \N}$ such that $u^i_k \to u^i$ in $W^{1,\psi}(\Omega)$. Let ${\bf u_k}:= (u^1_k, ... u^n_k)$. By Lemma \ref{lem:component_convergence}, ${\bf u_k} \to {\bf u}$ in $W^{1,\psi}(\Omega)$.
\end{itemize}
Having density of $C_c^{\infty}(\Omega)$ in $W^{1,\psi}_0(\Omega)$ in hand, the absence of Lavrentiev phenomenon follows as in the proof of Lemma \ref{lem:density_implies_lavr}.
\end{proof}

\appendix
\section{Supplementary material}
\subsection{Vitali convergence theorem}
\noindent In this section we recall a convergence result that is used several times in this paper. For the proof, see \cite[Theorem 4.5.4]{MR2267655}.
\begin{thm}\label{thm:vitali}
Let $(X,\mathcal{F},\mu)$ be a finite measure space (i.e. $\mu(X)<\infty$). Let $\{f_n\}_{n \in \N} \subset L^1(X,\mathcal{F},\mu)$ and $f$ be an $\mathcal{F}$-measurable function. Then, $f_n \to f$ in $L^1(X,\mathcal{F},\mu)$ if and only if $f_n \to f$ in measure and $\{f_n\}_{n \in \N}$ is uniformly integrable, i.e.
$$
\forall{\varepsilon>0} \, \exists{\delta>0} \, \forall{A \in \mathcal{F}} \, \qquad \mu(A) < \delta \implies \sup_{n \in \N} \int_{A} |f_n| \diff \mu < \varepsilon.
$$
\end{thm}
\noindent In the proof of the main results, we have applied the following corollary.
\begin{cor}\label{cor:vitali}
Let $(X,\mathcal{F},\mu)$ be a finite measure space (i.e. $\mu(X)<\infty$). Let $\{f_n\}_{n \in \N} \subset L^1(X,\mathcal{F},\mu)$ be a nonnegative sequence and $f$ be an $\mathcal{F}$-measurable function. Suppose that
\begin{enumerate}[label=(J\arabic*)]
    \item $f_n \to f$ in measure,
    \item there exists a sequence of functions $\{g_n\}_{n \in \N}$ convergent in $L^1(X,\mathcal{F},\mu)$ and function $h \in L^1(X,\mathcal{F},\mu)$ such that
    $$
    0 \leq f_n \leq g_n + h.
    $$
\end{enumerate}
Then, $f_n \to f$ in $L^1(X,\mathcal{F},\mu)$.
\end{cor}
\begin{proof}
In view of Theorem \ref{thm:vitali}, it is sufficient to prove that $\{f_n\}_{n \in \N}$ is uniformly integrable. To this end, for an arbitrary set $A$, we have
$$
\int_{A} |f_n| \diff \mu = \int_{A} f_n \diff \mu \leq \int_{A} g_n \diff \mu + \int_{A} h \diff \mu \leq \int_{A} |g_n + h| \diff \mu.
$$
Let $\varepsilon > 0$. As $\{g_n\}_{n \in \N}$ is convergent in $L^1(X,\mathcal{F},\mu)$, the same is true for $\{g_n + h\}_{n \in \N}$. It follows that $\{g_n +h\}_{n \in \N}$ is uniformly integrable. Therefore, there exists ${\delta}>0$ such that if $\mu(A) < {\delta}$, we have $\int_{A} |g_n + h| \diff \mu < \varepsilon$. It follows that
$$
\int_{A} |f_n| \diff \mu <  \varepsilon.
$$

\end{proof}

\subsection{Proof of Lemma \ref{lem:conv_mus_orlicz_spaces_Delta2}}\label{app:proof_of_orlicz_lemma}
\begin{proof}
The first equivalence in \ref{item:mus-or:1} follows directly from definition of the norm \eqref{eq:musor_norm} so in fact it is sufficient to prove that if $ \int_{\Omega} \psi(x,c|f(x)|) \diff x < \infty$ for some $c>0$ then $ \int_{\Omega} \psi(x,d|f(x)|) \diff x < \infty$ for all $d>0$. First, if $d < c$, this follows by convexity and Jensen's inequality:
\begin{equation}\label{eq:comp_cd}
\int_{\Omega} \psi(x,d|f(x)|) \diff x = \int_{\Omega} \psi\left(x,\frac{d}{c} \, c|f(x)| + 0 \right) \diff x \leq \frac{d}{c} \int_{\Omega} \psi(x,c|f(x)|) \diff x.
\end{equation}
If $d > c$, we find $k \in \N$ such that $d \leq 2^k \,c$ and apply \ref{ass:Nf_delta2} in Assumption \ref{ass:Nfunction}:
\begin{equation}\label{eq:comp_dc}
\int_{\Omega} \psi(x,d|f(x)|) \diff x \leq C_4^k \, \int_{\Omega} \psi\left(x,\frac{d}{2^k}|f(x)|\right) \diff x \leq  C_4^k \frac{d}{2^k c}  \int_{\Omega} \psi(x,c|f(x)|) \diff x.
\end{equation}
where we used the first part. \\

\noindent Concerning \ref{item:mus-or:2}, we first prove equivalence:
$$
\|f_n - f \|_{\psi} \to 0 \iff \int_{\Omega} \psi(x, c\,|f_n - f|) \diff x \to 0 \mbox{ for all } c>0.
$$
To prove ($\Rightarrow$) we fix $c>0$ and we note that there exists $n_c$ such that for all $n \geq n_c$ we have $c\,\|f_n - f \|_{\psi} < 1$. By definition \eqref{eq:musor_norm}, there exists a sequence $\{\delta_k\}_{k \in \N}$ convergent to 0 such that $c\,\|f_n - f \|_{\psi} + \delta_k < 1$ and 
$$
\int_{\Omega} \psi\left(x, \frac{c\,|f_n - f|}{c\, \|f_n - f \|_{\psi} + \delta_k } \right) \leq 1.
$$
Using convexity of $\psi$ and equality $\psi(x,0) = 0$ we obtain
$$
\int_{\Omega} \psi\left(x, {c\,|f_n - f|} \right) \leq \left(\|f_n - f \|_{\psi} + \delta_k\right)\, \int_{\Omega} \psi\left(x, \frac{c\,|f_n - f|}{ \|f_n - f \|_{\psi} + \delta_k } \right) \leq c\, \|f_n - f \|_{\psi} + \delta_k.
$$
Letting $k \to \infty$ (so that $\delta_k \to 0$) and $n \to \infty$ we conclude the proof. For $(\Leftarrow)$, we note that for each $c > 0$, there exists $n_c$ such that for all $n \geq n_c$ we have $ \int_{\Omega} \psi(x, c\,|f_n - f|) \diff x \leq 1$, i.e. $\|f_n - f\|_{\psi} \leq \frac{1}{c}$. The conclusion follows by letting $c \to \infty$. We are left to prove equivalence 
$$
\int_{\Omega} \psi(x, c\,|f_n - f|) \diff x \to 0 \mbox{ for all } c>0 \iff \int_{\Omega} \psi(x, c\,|f_n - f|) \diff x \to 0 \mbox{ for some } c>0.
$$
This follows from \eqref{eq:comp_cd} and \eqref{eq:comp_dc}. 
\\

\noindent To prove \ref{item:mus-or:3} we assume that $\int_{\Omega} \psi(x, |f_n - f|) \diff x \to 0$ and $\|f\|_{\psi} < \infty$ which implies $\int_{\Omega}\psi(x,|f|) \diff x < \infty$. First, thanks to the lower bound cf. \ref{ass:lowerbound} in Assumption \ref{ass:Nfunction} we have $m_{\psi}(|f_n - f|) \to 0$ in $L^1(\Omega)$ so that $f_n \to f$ in measure. Second, we can estimate
\begin{multline*}
0 \leq \psi(x, |f_n|) \leq \psi \left(x, \frac{1}{2}\,2|f_n-f| + \frac{1}{2}\,2|f| \right) \leq \\ \leq \frac{1}{2} \psi(x, 2|f_n-f|) + \frac{1}{2} \psi(x,2|f|) \leq 
\frac{C_4}{2} \psi(x,|f_n-f|) + \frac{C_4}{2} \psi(x,|f|). 
\end{multline*}
Corollary \ref{cor:vitali} implies that $\psi(x, |f_n|) \to \psi(x, |f|)$ in $L^1(\Omega)$ so in particular, $\int_{\Omega}\psi(x,|f_n|) \diff x \to \int_{\Omega}\psi(x,|f|) \diff x$.\\

\noindent Concerning \ref{item:mus-or:4}, in view of Vitali convergence theorem cf. Theorem \ref{thm:vitali}, it is sufficient to prove that the sequence $\{\psi(x, f_n - f)\}_n$ is uniformly integrable. Using convexity and $\Delta_2$ condition we obtain
\begin{multline*}
0 \leq \psi(x, |f_n - f|) \leq \psi \left(x, \frac{1}{2}\,2|f_n| + \frac{1}{2}\,2|f| \right) \leq \\ \leq \frac{1}{2} \psi(x, 2|f_n|) + \frac{1}{2} \psi(x,2|f|) \leq 
\frac{C_4}{2} \psi(x,|f_n|) + \frac{C_4}{2} \psi(x,|f|). 
\end{multline*}
It follows that $\psi(x, |f_n - f|) \to 0$ in $L^1(\Omega)$ and the conclusion follows from \ref{item:mus-or:2}.\\

\noindent Finally, to show \ref{item:mus-or:5}, we have to prove that if $\int_{\Omega} \psi(x,|f(x)|) \diff x < \infty$ then $\int_{\Omega} |f(x)|^p \diff x < \infty$. To this end, we estimate
$$
\int_{\Omega} |f(x)|^p \diff x = \int_{\Omega \cap \{|f| \leq \xi_0\} } |f(x)|^p \diff x + \int_{\Omega \cap \{|f| \geq \xi_0\} } |f(x)|^p \diff x \leq |\xi_0|^p + \int_{\Omega} \varphi(x,|f(x)|) \diff x
$$
where we used \ref{ass:Nf_pq} in Assumption \ref{ass:Nfunction}.

\end{proof}

\bibliographystyle{abbrv}
\bibliography{parpde_mo_discmodul}

\end{document}